\documentclass[1p,preprint,10pt]{elsarticle}

\usepackage{amsmath,amssymb,amsthm}

\usepackage{cleveref}
\usepackage[mathscr]{euscript}
\usepackage{url}

\usepackage{booktabs}

\usepackage{tikz}
\usetikzlibrary{shapes}
\usetikzlibrary{arrows}
\usetikzlibrary{matrix}
\usetikzlibrary{cd}

\newtheorem{theorem}{Theorem}[section]
\newtheorem{corollary}[theorem]{Corollary}
\newtheorem{lemma}[theorem]{Lemma}
\newtheorem{definition}[theorem]{Definition}
\newtheorem{proposition}[theorem]{Proposition}
\newtheorem{remark}[theorem]{Remark}

\newtheorem{algorithm}{Algorithm}
\newtheorem{example}[theorem]{Example}

\usepackage{color}
\definecolor{darkred}{rgb}{0.7,0,0}
\definecolor{darkgreen}{rgb}{0,0.47,0}
\definecolor{purple}{rgb}{0.6,0,0.5}

\newcommand{\tens}[1]{\boldsymbol{\mathcal{#1}}}
\newcommand{\tenselem}[1]{\mathcal{#1}}

\newcommand{\matr}[1]{\boldsymbol{#1}}

\newcommand{\vect}[1]{\boldsymbol{#1}}

\newcommand{\set}[1]{\mathscr{#1}}
\newcommand{\con}{\mathop{\bullet}\displaylimits}   
\newcommand{\T}{{\sf T}}        

\makeatletter
\newcommand{\rank}[1]{\mathop{\operator@font rank}\{#1\}}
\newcommand{\colrank}[1]{\mathop{\operator@font colrank}\{#1\}}
\newcommand{\krank}[1]{\mathop{\operator@font krank}\{#1\}}
\newcommand{\srank}[1]{\mathop{\operator@font srank}\{#1\}}

\newcommand{\trace}[1]{\mathop{\operator@font trace}\{#1\}}
\newcommand{\Diag}[1]{\mathop{\operator@font Diag}\{#1\}}    
\newcommand{\diag}[1]{\mathop{\operator@font diag}\{#1\}}    
\newcommand{\Span}[1]{\mathop{\operator@font Span}\{#1\}}    
\newcommand{\argmin}{\mathop{\operator@font argmin}}

\newcommand{\offdiag}[1]{\mathop{\operator@font offdiag}\{#1\}}    
\newcommand{\Proj}[2]{\mathop{\operator@font Proj_{#1}}{#2}}
\newcommand{\ProjGrad}[2]{\mathop{{\operator@font Proj} \nabla}#1(#2)}
\makeatother
\newcommand{\h}{\bar{h}}

\newcommand{\eqdef}{\stackrel{\sf def}{=}}
\newcommand{\RR}{\mathbb{R}}

\newcommand{\NN}{\mathbb{N}}
\newcommand{\LL}{\set{L}}
\newcommand{\PP}{\set{P}}

\newcommand{\ON}[1]{\set{O}_{#1}}
\newcommand{\SON}[1]{\set{SO}_{#1}}

\newcommand{\Gmat}[3]{\matr{G}^{(#1,#2,#3)}}

\newcommand{\contr}[1]{\con_{#1}}

\DeclareMathOperator*{\argmax}{\arg\!\max}

\begin{document}
\begin{frontmatter}
\author[g]{Jianze Li}\ead{lijianze@gmail.com}
\author[c]{Konstantin Usevich\corref{cor}}\ead{konstantin.usevich@univ-lorraine.fr}
\author[g]{Pierre Comon}\ead{pierre.comon@gipsa-lab.fr}
\cortext[cor]{Corresponding author}
\address[g]{Univ. Grenoble Alpes, CNRS, Grenoble INP, GIPSA-lab, Grenoble, France}
\address[c]{Universit\'{e} de Lorraine, CNRS, CRAN, Nancy, France}
\title{On approximate diagonalization of third order symmetric tensors by orthogonal transformations}

\begin{abstract}
In this paper,
we study the approximate  orthogonal diagonalization problem of third order symmetric tensors.
We define several classes of approximately diagonal tensors, including the ones corresponding to the  stationary points of this problem. 
We study  the  relationships between these classes, and other well-known objects, such as tensor Z-eigenvalue and Z-eigenvector.  
We also prove results on convergence of the cyclic Jacobi (or Jacobi CoM2) algorithm. 
\end{abstract}

\begin{keyword}symmetric tensors; orthogonally decomposable tensors;  approximate  tensor diagonalization; Jacobi-type algorithms; maximally diagonal tensors

\MSC[2010]15A69, 65F99, 90C30
\end{keyword}

\end{frontmatter}

\section{Introduction}

Arrays with more than two indices have become more and more important in the last two decades because of their usefulness in various fields,
including \emph{signal processing, numerical linear algebra} and \emph{data analysis} 
\cite{Cichocki15:review,comon2014tensors,Como10:book,kolda2009tensor,sidiropoulos2017tensor}. 
Admitting a common abuse of language, we shall refer to them as tensors, being understood that we are considering  the  associated multilinear \emph{forms} (and hence fully contravariant tensors) \cite{comon2014tensors}.
Real symmetric matrices can be diagonalized by orthogonal  transformations,  which is a key property leading to  the  spectral decomposition.
On the other hand, the orthogonal diagonalization of symmetric tensors has also been addressed, as an exact decomposition in \cite{robeva2016orthogonal,kolda2001orthogonal,kolda2015symmetric}, or as a low-rank approximation in \cite{Como92:elsevier,Como94:sp}.  In fact,  the  approximate orthogonal diagonalization of third and fourth order cumulant tensors  is in  the core of \emph{Independent Component Analysis}  \cite{Como92:elsevier,Como94:sp,Como94:ifac}, 
and finds many applications \cite{Como10:book}.
However, the latter  problem is much more difficult than the spectral decomposition of symmetric matrices
since  it is well known that  not every symmetric tensor can be diagonalized by orthogonal  transformations  \cite{robeva2016orthogonal,kolda2001orthogonal}. 

\textbf{Notation.} 
Let  $\RR^{m\times n\times p}\eqdef\RR^{m}\otimes\RR^{n}\otimes\RR^{p}$  be the linear space of third order real tensors and
$\set{S}_n\subseteq\RR^{n\times n\times n}$ be the set of symmetric ones,
whose entries do not change under any permutation of indices \cite{Comon08:symmetric,qi2017tensor}.
Let $\ON{n}\subseteq\RR^{n\times n}$ be the orthogonal group.
Let $\SON{n}\subseteq\RR^{n\times n}$ be the special orthogonal group,
that is,
the set of orthogonal matrices with determinant 1.
We denote by $\|\cdot\|$ the Frobenius norm of a tensor or a matrix,
or the Euclidean norm of a vector.
Tensor arrays, matrices, and vectors,  will be respectively denoted by bold calligraphic letters, e.g. $\tens{A}$, with bold uppercase letters, e.g. $\matr{M}$, and with bold lowercase letters, e.g. $\vect{u}$; corresponding entries will be denoted by $\tenselem{A}_{ijk}$, $M_{ij}$, and $u_i$.
 Operator $\contr{p}$ denotes contraction on the $p$th index of a tensor; when contracted with a matrix, it is understood that summation is always performed on the second index of the matrix. For instance, $[\tens{A}\contr{1}\matr{M}]_{ijk}=\sum_\ell \tenselem{A}_{\ell jk} M_{i\ell}$. 
When contraction  of a symmetric tensor  is performed on vectors, the subscript $p$ can be omitted.
For $\tens{A}\in\set{S}_n$ and a fixed set of indices $\{i,j\}$, $1\le i<j\le n$,
we denote by $\tens{A}^{(i,j)}$ the 2-dimensional subtensor obtained from  $\tens{A}$ by allowing its indices to vary in $\{i,j\}$ only. 
Similarly for $1\le i<j<k\le n$,
we denote by $\tens{A}^{(i,j,k)}$ the 3-dimensional subtensor obtained by allowing indices of $\tens{A}$ to vary in $\{i,j,k\}$ only.
The identity matrix of size $n$  is denoted by $\matr{I}_n$, and its columns by $\vect{e}_i$, $1\le i\le n$, which form the canonical orthonormal basis.

\textbf{Contribution.} 
We formulate the  \emph{approximate  orthogonal symmetric tensor diagonalization problem}  as the maximization of  diagonal terms \cite{Comon07:tensor}.
More precisely, let $\tens{A}\in\set{S}_n$, $\matr{Q}\in\SON{n}$,  and
\[
\tens{W} =  \tens{A}\contr{1}\matr{Q}^{\T}\contr{2} \matr{Q}^{\T}\contr{3} \matr{Q}^{\T}.
\]
 This  problem is to find
\begin{equation}\label{pro-ortho-diagonal}
\matr{Q}_{*} = \argmax_{\matr{Q}\in\SON{n}}f(\matr{Q}),
\end{equation}
where
\begin{equation}\label{cost-function-f-q}
f(\matr{Q}) \eqdef \|\diag{\tens{W}}\|^2 =  \sum_{i=1}^{n}\tenselem{W}_{iii}^2.
\end{equation}
Methods based on Jacobi rotations (e.g., the well-known Jacobi CoM2 algorithm  \cite{Como92:elsevier,Como94:sp,Como94:ifac}) are widely used in practice  \cite{Como10:book,martin2008jacobi} to solve problem \eqref{pro-ortho-diagonal}.  
These methods aim at making a symmetric tensor as diagonal as possible by successive Jacobi rotations.
They are particularly attractive due to the low computational cost of iterations.  Other popular methods include Riemannian optimization methods \cite{absil2009optimization} that alternate between descent steps and retractions. 
The above methods are typically known to converge (globally or locally) to stationary points
 \cite{absil2009optimization,LUC2017globally}, though the convergence of the original Jacobi CoM2 method has not been studied. 
 
The main goal of this paper is to quantify the notion of approximate diagonality, by introducing several classes of approximately diagonal tensors and studying  the  relationships between them. 
These classes include \emph{stationary diagonal} tensors,
\emph{Jacobi diagonal} tensors,
\emph{locally maximally diagonal} tensors,
 \emph{maximally diagonal} tensors,
\emph{generally maximally diagonal} tensors and \emph{pseudo diagonal} tensors.
We characterize 
(i)  the class of Jacobi diagonal tensors by the stationary diagonal ratio,  and 
(ii)  the  orbit  of pseudo diagonal tensors by Z-eigenvalue and Z-eigenvectors.  Moreover, we study 
(iii) the class of locally maximally diagonal tensors based on Riemannian Hessian.
We  show that this class is not equal to the class of Jacobi diagonal tensors,
and thus Jacobi-type algorithms may converge to a saddle point of \eqref{cost-function-f-q}. We also study 
(iv) whether a symmetric tensor is maximally diagonal if and only if it is generally maximally diagonal.
Several problems related to low rank orthogonal approximation are proved to be equivalent to  the fact that
these two classes are equal when the dimension is greater than $2$.
We present a counterexample to these equivalent problems based on the decomposition of orthogonal matrices.
Moreover, we prove  a result that can be seen as an orthogonal analogue of  the so-called  \emph{Comon's Conjecture}  \cite{zhang2016comon}.
 The second goal  of this paper  is to study the convergence properties of the original Jacobi CoM2 algorithm \cite{Como94:ifac}.

\textbf{Organization.} The paper is organized as follows.
In \cref{sec:opt_prob}, we recall basic properties of the cost function, introduce notation for derivatives, and present the scheme of Jacobi-type algorithms.
In \cref{sec:classes}, 
we define the classes of approximately diagonal tensors,  which are  considered in this paper. 
 Some basic  relationships  between  these classes are shown.
The stationary diagonal ratio is introduced, and the  orbit  of pseudo diagonal tensors is studied. In \cref{sec:LMD},
we study the class of locally maximally diagonal tensors  using  Riemannian Hessian.
In \cref{sec:GMD},
we study the relationship between maximally diagonal tensors and generally maximally diagonal tensors.  
\Cref{sec:convergence} contains results on convergence of the Jacobi CoM2 algorithm. 
Finally, \ref{proofs} contains the remaining  proofs. 

\section{Optimization problem: properties and algorithms}\label{sec:opt_prob}
\subsection{Riemannian gradient and stationary points}
 First, we recall that the Riemannian gradient of \eqref{cost-function-f-q}  \cite[\S 4.1]{LUC2017globally},  is, by definition,
\begin{equation}\label{eq-prj-grad}
\ProjGrad{f}{\matr{Q}}=\matr{Q}\matr{\Lambda}(\matr{Q}),
\end{equation}
where  $\matr{\Lambda}(\matr{Q})$  is the matrix with entries
\begin{equation}\label{eq:Projgrad-general}
\Lambda(\matr{Q})_{k,l}= 3( \tenselem{W}_{lll}\tenselem{W}_{llk} -\tenselem{W}_{lkk}\tenselem{W}_{kkk}).
\end{equation}
 The matrix $\matr{Q}$ is a stationary point of \eqref{cost-function-f-q} if and only if $\ProjGrad{f}{\matr{Q}} =0$.
A local maximum point of \eqref{cost-function-f-q}, of course, is a stationary point.
A reasonable local optimization  algorithm should at least  converge to a stationary point.

\subsection{Elementary rotations and Jacobi-type algorithms}
Let $(i,j)$ be a pair of indices with $1 \le i < j \le n$.
We denote the  \emph{Givens rotation}  (by an angle $\theta\in\RR$) matrix to be
\[
\Gmat{i}{j}{\theta} =
\begin{bmatrix}
1 &       & & & &&\\
  &\ddots & & & & \matr{0} &\\
  & & \cos \theta & & -\sin\theta && \\
  & & & \ddots & &&\\
  & & \sin \theta & & \cos\theta && \\
  & \matr{0} & & & & \ddots &\\
 &       & & & &&1
\end{bmatrix},
\]
i.e., the matrix defined by
\[
(\Gmat{i}{j}{\theta})_{k,l} =
\begin{cases}
1, &  k = l, k\not\in \{i,j\}, \\
\cos{\theta}, &  k = l, k\in \{i,j\}, \\
\sin{\theta}, &  (k,l) = (j,i),\\
-\sin{\theta}, & (k,l) = (i,j),\\
0, &  \text{otherwise} \\
\end{cases}
\]
for $1 \le k,l \le n$.

Jacobi-type algorithms proceed by successive optimization of  the cost function with respect to elementary rotations, summarized in the following scheme.
\begin{algorithm}\label{alg:jacobi-type}
{\bf Input:} $\tens{A}\in\set{S}_n$ and $\matr{Q}_{0}=\matr{I}_{n}$.\\
{\bf Output:}  a sequence of iterations $\{\matr{Q}_k:k\in\NN\}$. 
\begin{itemize}
\item {\bf For} $k=1,2,\ldots$ until a stopping criterion is satisfied  do: 
  \begin{itemize}
  \item\quad Choose the pair $(i_k,j_k)$ according to a certain pair selection  rule.
  \item\quad Compute the angle $\theta_{k}^{*}$ that maximizes  the function
  \begin{equation}\label{eq-def-h-k}
   h_k(\theta)  \eqdef  f(\matr{Q}_{k-1}\Gmat{i_k}{j_k}{\theta}).
  \end{equation}
  \item\quad  Update $\matr{Q}_k = \matr{Q}_{k-1} \Gmat{i_k}{j_k}{\theta_{k}^{*}}$.
  \end{itemize}
\item {\bf End for}
\end{itemize}
\end{algorithm}
The algorithm is similar in spirit to block-coordinate descent. 
Important differences are: the coordinate system is changing at every iteration, and, for each elementary rotation, the global maximum is achieved.
Recently, local and global convergence to stationary points \cite{IshtAV13:simax,LUC2017globally}  has  been established for variants of Algorithm~\ref{alg:jacobi-type}. Apart from Jacobi-type algorithms, Jacobi rotations are also very useful in computing  the  fast retractions \cite[p. 58]{absil2009optimization} in Riemannian optimization methods \cite{absil2009optimization}.

%

\subsection{Directional derivatives}
 We  introduce some useful notation  that will be used throughout the paper. 
\begin{definition}
Let $\tens{A}\in\set{S}_n$ and $1\le i<j \le n$.
Define
\begin{align*}
&{d}_{i,j}(\tens{A}) \eqdef \tenselem{A}_{iii}\tenselem{A}_{iij}-\tenselem{A}_{ijj}\tenselem{A}_{jjj},\\
&\omega_{i,j}(\tens{A}) \eqdef \tenselem{A}_{iii}^2+\tenselem{A}_{jjj}^2-3\tenselem{A}_{iij}^2-3\tenselem{A}_{ijj}^2
-2\tenselem{A}_{iii}\tenselem{A}_{ijj}-2\tenselem{A}_{iij}\tenselem{A}_{jjj}.
\end{align*}
\end{definition}

In order to simplify notation,  we denote functions \eqref{eq-def-h-k} with $\matr{Q}_{k-1} = \matr{I}_n$ as
\begin{equation*}\label{eq-defi-ij-h}
\h_{i,j}(\theta)  \eqdef \|\diag{\tens{A}\contr{1}(\Gmat{i}{j}{\theta})^{\T}\contr{2}(\Gmat{i}{j}{\theta})^{\T}
\contr{3}(\Gmat{i}{j}{\theta})^{\T}}\|^2
\end{equation*}
for $1\le i<j\le n$.
 Then  it holds that \cite[Lemma 5.7]{LUC2017globally}
\begin{equation}\label{eq-d-omega}
\h_{i,j}^{'}(0) = 6 {d}_{i,j}(\tens{A})\quad\text{and}\quad\h_{i,j}^{''}(0) = -6 \omega_{i,j}(\tens{A}).
\end{equation}

\section{Classes of approximately diagonal tensors}\label{sec:classes}
\subsection{Definitions of classes}\label{sec:defs}
In this subsection, we define several classes of third order symmetric tensors.
Some of them are related to  the  stationary points of  \eqref{cost-function-f-q}  or the points where Algorithm~\ref{alg:jacobi-type} may stop.
For simplification, we look at the derivatives of  \eqref{cost-function-f-q}  at $\matr{Q} = \matr{I}_n$.

\begin{definition}
(i) Let $\tens{A},\tens{B}\in\set{S}_n$.
Then $\tens{A}$ is  \emph{orthogonally similar} \cite{qi2017tensor,qi2009z}  to $\tens{B}$ if there exists $\matr{Q}\in\ON{n}$ such that
\[
\tens{B} = \tens{A}\contr{1}\matr{Q}\contr{2}\matr{Q}\contr{3}\matr{Q}.
\]
(ii) Let $\set{C}\subseteq\set{S}_n$ be a subset.
Define the \emph{orbit}\footnote{Classically, the notion of orbit is defined for a single element (e.g., $\tens{C} \in \set{S}_n$). In this paper, we use the word ``orbit'' as a shorthand for saying ``the action of  $\set{O}_n$  on $\set{C}$''.} of $\set{C}$ to be: 
\[
 \set{O}(\set{C})  \eqdef\{\tens{A}\contr{1}\matr{Q}\contr{2}\matr{Q}\contr{3}\matr{Q}, 
\ \tens{A}\in\set{C},\ \matr{Q}\in\ON{n}\}.
\]
\end{definition}

\begin{definition}
We denote by  $\set{D}_n$  the set of diagonal tensors in $\set{S}_n$,  and  $\set{O}(\set{D}_n)$  
the set of orthogonally decomposable tensors  (referred to as ``odeco''   in \cite{robeva2016orthogonal}). 
More precisely, any  $\tens{A}\in\set{O}(\set{D}_n)$  can be decomposed as
\[
\tens{A} = \sum_{k=1}^{n}\lambda_{k} \, \vect{u}_{k}\otimes \vect{u}_{k}\otimes \vect{u}_{k}
\]
where
$\lambda_{k}\in\RR$ and $\vect{u}_{1},\cdots \vect{u}_{n}\in\RR^{n}$ form an orthonormal basis.
\end{definition}

\begin{definition}
Let $\tens{A}\in\set{S}_n$.
The class of \emph{pseudo diagonal} tensors is defined to be
\[
 \set{PD}_n  \eqdef \{\tens{A}: \tenselem{A}_{ijj} = \tenselem{A}_{iij} = 0,\ \text{for any}\ 1\le i<j\le n\}.
\]
\end{definition}

\begin{remark}\label{notation-D-ED}
It is clear that
\[
\set{D}_n\subseteq\set{PD}_n\quad\text{and}\quad\set{O}(\set{D}_n)\subseteq\set{O}(\set{PD}_n).
\]
In \cref{subsec-z-eigen}, we will give characterizations of  $\set{PD}_n$  and  $\set{O}(\set{PD}_n)$  from the perspective of tensor spectral theory. Besides, it is well known  that  $\set{O}(\set{D}_n)\subsetneqq\set{S}_n$, 
that is,
not every symmetric tensor can be diagonalized by orthogonal  transformations  \cite{kolda2001orthogonal,robeva2016orthogonal}.
\end{remark}

\begin{definition}\label{definition-classes-2}
Let $\tens{A}\in\set{S}_n$.\\
(i) The class of \emph{stationary diagonal} tensors is defined to be
\[
 \set{SD}_n  \eqdef \{\tens{A}: {d}_{i,j}(\tens{A}) = 0,\ \text{for any}\ 1\le i<j\le n\}.
\]
(ii) The class of \emph{Jacobi diagonal} tensors is defined to be
\[
 \set{JD}_n  \eqdef \{\tens{A}:
0 \in \argmax\limits_{\theta\in\RR}\h_{i,j}(\theta),\ \text{for any}\ 1\le i<j\le n\}.
\]
(iii) The class of \emph{locally Jacobi diagonal} tensors is defined to be
\[
 \set{LJD}_n  \eqdef \{\tens{A}:
0\ \text{is a  local  maximum point of}\ \h_{i,j}(\theta),\ \text{for any}\ 1\le i<j\le n\}.
\]
\end{definition}

\begin{remark}\label{remark-stationary}
From \eqref{eq:Projgrad-general}, it follows that  $\tens{A}\in\set{SD}_n$  if and only if $\ProjGrad{f}{\matr{I}_{n}}=0$ in \eqref{eq-prj-grad}.
In other words,
 $\tens{A}\in\set{SD}_n$  if and only if  $\matr{I}_{n}$ is a stationary point of \eqref{cost-function-f-q}.
Moreover,
it can be seen that Algorithm~\ref{alg:jacobi-type} stops at $\tens{A}$ if  $\tens{A}\in\set{JD}_n$. 
This is the reason why we call the tensors in  $\set{SD}_n$  and  $\set{JD}_n$  stationary diagonal  and Jacobi diagonal respectively. 
\end{remark}

\begin{lemma}\label{lemma-C3-h}
Let $\tens{A}\in\set{S}_n$.
The following are equivalent.\\
(i)  $\tens{A}\in\set{JD}_n$. \\
(ii)  $\tens{A}\in\set{LJD}_n$. \\
(iii) ${d}_{i,j}(\tens{A})=0$ and
$\omega_{i,j}(\tens{A})\geq 0$
for any $1\le i<j\le n$.
\end{lemma}
\begin{proof}
(i)$\Rightarrow$(ii) is clear.
(ii)$\Rightarrow$(iii)   follows from   \eqref{eq-d-omega}.
 Let us prove  (iii)$\Rightarrow$(i).
We have
\begin{equation}\label{eq-increament-hij}
 \h_{i,j}(\theta)-\h_{i,j}(0)  =\frac{3}{(1+x^2)^2}
(2{d}_{i,j}(\tens{A})(x-x^3)-\omega_{i,j}(\tens{A})x^2)
\end{equation}
for  $x = \tan(\theta)$,  any $1\le i<j\le n$ by  \cite[Eq. (22)]{LUC2017globally} (see also  \eqref{equa-increase-h}).
Note that  $\h_{i,j}(\theta) - \h_{i,j}(0) \equiv 0$  if
${d}_{i,j}(\tens{A})= \omega_{i,j}(\tens{A})= 0$.
If ${d}_{i,j}(\tens{A})=0$ and
$\omega_{i,j}(\tens{A})\geq 0$,
then  $\h_{i,j}(\theta)$  reaches its maximum value at $\theta=0$, by \eqref{eq-increament-hij}.
It follows that $\tens{A}\in\set{JD}_n$. 
\end{proof}

\begin{definition}\label{definition-classes-3}
Let $\tens{A}\in\set{S}_n$ {and} $f$ be as in \eqref{cost-function-f-q}.\\
(i) The class of \emph{maximally diagonal} tensors is defined to be
\[
\set{MD}_n \eqdef \{\tens{A}:
\matr{I}_{n}  \in \argmax\limits_{\matr{Q}\in\SON{n}}f(\matr{Q})\}.
\]
(ii) The class of \emph{locally maximally diagonal} tensors is defined to be
\[
 \set{LMD}_n  \eqdef \{\tens{A}:
\matr{I}_{n}\ \text{is a  local  maximum point of}\ f(\matr{Q})\}.
\]
(iii) The class of \emph{generally maximally diagonal} tensors is defined to be
\[
 \set{GMD}_n  \eqdef \{\tens{A}:
(\matr{I}_{n},\matr{I}_{n},\matr{I}_{n}) \in \argmax\limits_{\matr{P},\matr{Q},\matr{R}\in\SON{n}}\mathcal{F}(\matr{P},\matr{Q},\matr{R})\},
\]
 where 
\begin{equation}\label{eq-defi-f-pqr}
\mathcal{F}(\matr{P},\matr{Q},\matr{R})\eqdef
\|\diag{\tens{A}\contr{1}\matr{P}^{\T}\contr{2}\matr{Q}^{\T}\contr{3}\matr{R}^{\T}}\|^2.
\end{equation}
\end{definition}

\begin{remark}
Note that $\ON{n}\subseteq\RR^{n\times n}$ is a compact submanifold and \eqref{cost-function-f-q} is continuous. Since \eqref{cost-function-f-q} takes the same maximum on $\ON{n}$ and $\SON{n}$, we get that $\set{O}(\set{MD}_n) = \set{S}_n.$ 
Note that $\set{MD}_n\subseteq\set{LMD}_n$.
It follows that $\set{O}(\set{LMD}_n) = \set{S}_n.$ 
In other words, 
for any $\tens{A}\in\set{S}_n$,
there  exist  $\matr{Q}_{*}$ and $\matr{Q}_{**}$ in $\set{SO}_n$ such that
\[
 \tens{A} \contr{1} \matr{Q}_{*}\contr{2} \matr{Q}_{*}\contr{3} \matr{Q}_{*}\in \set{LMD}_n\quad \text{and}\quad  \tens{A} \contr{1} \matr{Q}_{**}\contr{2} \matr{Q}_{**}\contr{3} \matr{Q}_{**}\in \set{MD}_n,
\]
respectively.
How to find $\matr{Q}_{*}$ or $\matr{Q}_{**}$ is the  goal  of problem \eqref{pro-ortho-diagonal}.
\end{remark}

\subsection{Basic  relationships}
The tensor classes defined in   \cref{sec:defs}  have the following relationships.
The first row  and column  denote the corresponding  orbits,  i.e.,  arrows stand for the action of $\set{O}_n$.
\[
\begin{tikzcd}
& \set{O}(\set{D}_n) \arrow[r, phantom, "\subseteq" description] & \set{O}(\set{PD}_n) \arrow[r, phantom, "\subsetneqq" description]  & \set{S}_n&&\\
\set{O}(\set{D}_n) \arrow[d, phantom, "{\rotatebox[origin=c]{270}{$\subseteq$}}" description] & \set{D}_n\arrow{l}  \arrow{u}\arrow[r, phantom, "\subseteq" description] \arrow[d, phantom, "{\rotatebox[origin=c]{270}{$\subsetneqq$}}" description] & \set{PD}_n \arrow{u} \arrow[r, phantom, "\subsetneqq" description]  & \set{JD}_n\arrow{u} \arrow[d, phantom, "{\rotatebox[origin=c]{270}{$=$}}" description]&& \\
\set{O}(\set{GMD}_n)\arrow[d, phantom, "{\rotatebox[origin=c]{270}{$\subseteq$}}" description] & \set{GMD}_n \arrow[d, phantom, "{\rotatebox[origin=c]{270}{$\subsetneqq$}}" description] \arrow{l} &  & \set{JD}_n\arrow[r, phantom, "\subsetneqq" description] \arrow[d, phantom, "{\rotatebox[origin=c]{270}{$=$}}" description] & \set{SD}_n\arrow[r, phantom, "\subsetneqq" description] & \set{S}_n\\
\set{S}_n & \set{MD}_n\arrow{l}  \arrow[r, phantom, "\subseteq" description] & \set{LMD}_n \arrow[r, phantom, "\subseteq" description]  & \set{LJD}_n&& \\
\end{tikzcd}
\]
\begin{remark}\label{remark-relationship-2}
Most of the above relationships are easy to get by Definition~\ref{definition-classes-2}
and Definition~\ref{definition-classes-3}.
We only derive some of them for $\set{S}_2$,
which are not obvious.\\
(i)
Note that $\SON{2}$ coincides with the set of Givens rotations.
We see that
\[
\set{MD}_2=\set{LMD}_2=\set{JD}_2=\set{LJD}_2
\]
by Lemma~\ref{lemma-C3-h}.
It will be shown that $\set{GMD}_2=\set{MD}_2$ 
in Theorem~\ref{equal-dim-2}.
It follows that
\[
\set{GMD}_2=\set{MD}_2=\set{LMD}_2=\set{JD}_2=\set{LJD}_2.
\]
(ii)  $\set{PD}_n$ and $\set{JD}_n$  will be characterized in
Remark~\ref{remark-SD-ratio} and Theorem~\ref{theorem-gamma-charac}.
It follows by these characterizations that $\set{PD}_2\varsubsetneqq\set{JD}_2$.\\
(iii) Note that $\set{D}_2=\set{PD}_2$.
It follows by (i) and (ii) that
\[\set{D}_2\varsubsetneqq\set{GMD}_2.\]
(iv) By Theorem~\ref{theorem-gamma-charac},
we see that $\set{JD}_2\varsubsetneqq\set{SD}_2$ .\\
(v) Note that $\set{D}_2=\set{PD}_2$ 
and
$\set{O}(\set{D}_2)\subsetneqq\set{S}_2$
by Remark~\ref{notation-D-ED}.
We have that
\[\set{O}(\set{PD}_2)\subsetneqq\set{S}_2.\]
\end{remark}

\subsection{Stationary diagonal ratio}\label{subsection-charac-C3}

In this subsection,
we define the stationary diagonal ratio for the tensors in  $\set{SD}_n$, 
which can be used to characterize $\set{JD}_n$ and $\set{PD}_n$. 

\begin{definition}\label{re-defi-index}
Let $\tens{A}\in\set{SD}_n$  and $1\le i<j\le n$.
The  \emph{stationary diagonal ratio},
denoted by $\gamma_{ij}$,
is defined as follows.
\[
\gamma_{ij} \eqdef
\begin{cases}
0, &  \text{if}\ \tens{A}^{(i,j)}= \mathbf{0};\\
\infty, &    \text{if}\ \tenselem{A}_{iii}=\tenselem{A}_{jjj}=0\quad\text{and}\quad\tenselem{A}^2_{ijj} +\tenselem{A}^2_{iij}\neq0;\\
\end{cases}
\] 
 otherwise, $\gamma_{ij}$ is the {(unique)} number such that
\[
\begin{pmatrix}\tenselem{A}_{ijj} \\ \tenselem{A}_{iij} \end{pmatrix}  = \gamma_{ij}\begin{pmatrix}\tenselem{A}_{iii}\\\tenselem{A}_{jjj}\end{pmatrix}.
\]
\end{definition}

\begin{remark}\label{remark-SD-ratio}
Let $\tens{A}\in\set{SD}_n$. 
Then $\tens{A}\in\set{PD}_n$  if and only if $\gamma_{ij}=0$ for any $1 \le i<j \le n$.
\end{remark}
\begin{theorem}\label{theorem-gamma-charac}
Let  $\tens{A}\in\set{SD}_n$.  
Then  $\tens{A}\in\set{JD}_n$  if and only if $\gamma_{ij}\in[-1,1/3]$ for any $1\le i<j\le n$.
\end{theorem}
\begin{proof}
Note that
$\tens{A}\in\set{JD}_n$  if and only if
${d}_{i,j}(\tens{A})=0$ and
$\omega_{i,j}(\tens{A})\geq 0$
for any $1\le i<j\le n$ by Lemma~\ref{lemma-C3-h}.
We only need to show that $\omega_{i,j}(\tens{A})\geq 0$ if and only if $\gamma_{ij}\in[-1,1/3]$.
If $\gamma_{ij}=\infty$,
then $\omega_{i,j}(\tens{A})< 0$.
If $\gamma_{ij}<\infty$,
{by Definition~\ref{re-defi-index}},
we have that
\[
-\omega_{i,j}(\tens{A}) = (3\gamma_{ij}^2+2\gamma_{ij}-1)(\tenselem{A}_{iii}^2+\tenselem{A}_{jjj}^2).
\]
It follows that $\omega_{i,j}(\tens{A})\geq0$
if and only if $\gamma_{ij}\in[-1,1/3].$
\end{proof}

\subsection{Orbit of the pseudo diagonal tensors}
\label{subsec-z-eigen}
\subsubsection{Characterization}
In this subsection,
we characterize the  orbit  of pseudo diagonal tensors
based on the Z-eigenvalue and Z-eigenvectors  defined in \cite{qi2017tensor}. 

\begin{definition}
Let $\tens{A}\in\set{S}_n$
and $\lambda\in\RR$.
If $\lambda$ satisfies 
\[
\tens{A}\mathop{\bullet}\vect{u}\mathop{\bullet}\vect{u} = \lambda \,\vect{u}
\] 
for a unit vector  $\vect{u}\in\RR^{n}$. 
Then $\lambda$ is called a  \emph{Z-eigenvalue}  \cite{qi2017tensor}  of $\tens{A}$.
This vector is called the  \emph{Z-eigenvector}  associated with $\lambda$.
\end{definition}

\begin{remark}\label{remark-EPD}
Let $\tens{A}, \tens{B}\in\set{S}_n$.
If $\tens{A}$ is orthogonally similar to $\tens{B}$,
then $\tens{A}$ and $\tens{B}$ have the same Z-eigenvalues \cite[Thm 2.20]{qi2017tensor}. 
In fact,
if
\[\tens{A} = \tens{B}\contr{1}\matr{Q}^{\T}\contr{2}\matr{Q}^{\T}\contr{3}\matr{Q}^{\T}\quad\text{and}\ \
\tens{A}\mathop{\bullet}\vect{u}\mathop{\bullet}\vect{u} = \lambda\,\vect{u}\]
for $\lambda\in\RR$ and a unit vector $\vect{u}\in\RR^{n}$,
then 
$\tens{B}\mathop{\bullet}{(\matr{Q}\vect{u})}\mathop{\bullet}{(\matr{Q}\vect{u})} = \lambda \,{\matr{Q}\vect{u}}.$ 
\end{remark}



\begin{theorem}\label{theorem-z-PD}
Let $\tens{A}\in\set{S}_n$.
 We have  two necessary and sufficient conditions below: \\
(i) $\tens{A}\in\set{PD}_n$  if and only if
$\{\vect{e}_{i}:  1\le i\le n\}$ is a set of Z-eigenvectors. This is equivalent~to
\[\tens{A}\mathop{\bullet}\vect{e}_i\mathop{\bullet}\vect{e}_i\mathop{\bullet}\vect{e}_j = 0\]
for any $1\le i\neq j\le n$.\\
(ii) $\tens{A}\in\set{O}(\set{PD}_n)$ 
if and only if there exists an orthonormal set of Z-eigenvectors $\{\vect{u}_i:1\le i\le n\}$. This is equivalent to 
\[
\tens{A}\mathop{\bullet}\vect{u}_i\mathop{\bullet}\vect{u}_i\mathop{\bullet}\vect{u}_j = 0
\] 
for any $1\le i\neq j\le n$.
In this case, $\tens{A}\contr{1}\matr{Q}_{*}^{\T}\contr{2}\matr{Q}_{*}^{\T}\contr{3}\matr{Q}_{*}^{\T}\in\set{PD}_n $ {for
 $\matr{Q}_{*} = [\vect{u}_1,\cdots,\vect{u}_n]$.}
 \end{theorem}

\begin{proof}
(i) By definition,  $\tens{A}\in\set{PD}_n$  if and only if $\tens{A}\mathop{\bullet}\vect{e}_i\mathop{\bullet}\vect{e}_i\mathop{\bullet}\vect{e}_j = 0$ for any $1\le i\neq j\le n$.
But if  $\tens{A}\mathop{\bullet}\vect{e}_i\mathop{\bullet}\vect{e}_i$ is orthogonal to every $\vect{e}_j$, $j\neq i$, it must be collinear to $\vect{e}_i$, which means
\[
\tens{A}\mathop{\bullet}\vect{e}_i\mathop{\bullet}\vect{e}_i = \lambda \vect{e}_i
\]
for some nonzero $\lambda$, which turns out to yield $\lambda=\tens{A}\mathop{\bullet}\vect{e}_i\mathop{\bullet}\vect{e}_i\mathop{\bullet}\vect{e}_i$. 
\\
(ii)  The second result  follows from (i) and Remark~\ref{remark-EPD}.
\end{proof}

\subsubsection{Relationship with orthogonally decomposable tensors}
\begin{example}\label{exam-not-equal-D-PD-1}
We present an example to show that $\set{O}(\set{D}_n)\subsetneqq\set{O}(\set{PD}_n)$
for $\set{S}_n$.
Let
\begin{align*}
\tens{A} =\,& \vect{e}_{1}\otimes \vect{e}_{2}\otimes \vect{e}_{3} + \vect{e}_{1}\otimes \vect{e}_{3}\otimes \vect{e}_{2} + \vect{e}_{2}\otimes \vect{e}_{3}\otimes \vect{e}_{1} \\
+\,& \vect{e}_{2}\otimes \vect{e}_{1}\otimes \vect{e}_{3} + \vect{e}_{3}\otimes \vect{e}_{1}\otimes \vect{e}_{2} + \vect{e}_{3}\otimes \vect{e}_{2}\otimes \vect{e}_{1}.
\end{align*}
 It is easy to see   that $\tens{A}\in\set{O}(\set{PD}_3)$.
On the other hand, it is known \cite[Prop. 3.1 and 4.3]{carlini2012solution} that the symmetric tensor rank is
\[
\srank{\tens{A}} = 4,
\]
hence $\tens{A}$ cannot be in  $\set{O}(\set{D}_3)$ (otherwise it would have rank  at most $3$).
\end{example}

\begin{proposition}\label{theorem-z-PD-D}
(i) Let $\tens{A}\in\set{PD}_n$.
Then $\tens{A}\in\set{D}_n$ if and only if 
\[
\tens{A}\mathop{\bullet}\vect{e}_i\mathop{\bullet}\vect{e}_j \in \text{span}\{\vect{e}_i,\vect{e}_j\}
\] 
for any $1\le i\neq j\le n$.\\
(ii) Let $\tens{A}\in\set{O}(\set{PD}_n)$.
Let $\{\vect{u}_i: 1\le i\le n\}$ be the set of  orthonormal Z-eigenvectors, proved to exist  in Theorem~\ref{theorem-z-PD}~(ii).
Then $\tens{A}\in\set{O}(\set{D}_n)$
if and only if 
\[
\tens{A}\mathop{\bullet}\vect{u}_i\mathop{\bullet}\vect{u}_j \in \text{span}\{\vect{u}_i,\vect{u}_j\}
\] 
for any $1\le i\neq j\le n$.
\end{proposition}
\begin{proof}
First note that
$\tens{A}\mathop{\bullet}\vect{e}_i\mathop{\bullet}\vect{e}_j \in \text{span}\{\vect{e}_i,\vect{e}_j\}$ 
for any $1\le i< j\le n$ if and only if $\tenselem{A}_{ijk} = 0$ for any  $1\leq i< j < k\leq n$.
Then (i) is proved. Next, (ii)  follows from (i) and Remark~\ref{remark-EPD}.
\end{proof}

\section{Locally maximally diagonal tensors}\label{sec:LMD}

Even if Givens rotations span $\set{SO}_n$, it is not obvious that a sequence of optimally chosen Givens rotations will find the optimal orthogonal transform in $\set{SO}_n$. In other words, we know that $\set{LMD}_n \subseteq\set{LJD}_n$,  but the converse may not be true.  This motivates the comparison between  $\set{LJD}_n$  and  $\set{LMD}_n$. 

\subsection{Riemannian Hessian}
In this subsection,
we study the conditions that a tensor in $\set{S}_n$
is locally maximally diagonal based on the Riemannian Hessian  \cite{absil2009optimization,Absil2013,edelman1998geometry}.

\begin{lemma}\label{lemma-Riemannian-Hessian}
Let $\tens{A}\in\set{S}_n$ and $f$ be as in \eqref{cost-function-f-q}.
Let  $\text{T}_{\matr{Q}}\ON{n}$
be the tangent vector space at  $\matr{Q}$; it contains matrices of the form $\matr{Q}\matr{\Delta}$, where $\matr{\Delta}$ are  skew-symmetric  matrices satisfying $\matr{\Delta}^{\T}=-\matr{\Delta}$.  We denote   
\begin{align*}
&\tens{U} =  \tens{A} \contr{3} \matr{Q}^{\T},\quad\
\tens{V} =  \tens{A} \contr{2} \matr{Q}^{\T}\contr{3} \matr{Q}^{\T},\\
&\tens{X} =  \tens{V} \contr{1} (\matr{Q}\matr{\Delta})^{\T},\quad
\tens{Y} =  \tens{U} \contr{1} (\matr{Q}\matr{\Delta})^{\T}\contr{2} (\matr{Q}\matr{\Delta})^{\T},\quad
\tens{Z} = \tens{V} \contr{1} (\matr{Q}\matr{\Delta}^2)^{\T}.
\end{align*}
Let $\text{Hess} f(\matr{Q})$ be the Riemannian Hessian of $f$ at $\matr{Q}$.
Then $\text{Hess} f(\matr{Q})(\matr{\Delta}_1,\matr{\Delta}_2)$ is a bilinear form  defined on $\text{T}_{\matr{Q}}\ON{n}$.
We have:  
\begin{align*}
\text{Hess}f(\matr{Q})(\matr{Q}\matr{\Delta},\matr{Q}\matr{\Delta})
= 6 \sum_{j}  (3\tenselem{X}_{jjj}^2+2\tenselem{Y}_{jjj}\tenselem{W}_{jjj} - \tenselem{Z}_{jjj}\tenselem{W}_{jjj}). 
\end{align*}
\end{lemma}
\begin{proof}
By \cite[eqn. (2.55)]{edelman1998geometry}, 
it can be calculated that
\begin{align}
&\text{Hess}f(\matr{Q})(\matr{Q}\matr{\Delta},\matr{Q}\matr{\Delta})\notag\\
&= \sum\limits_{i,j,k,l}\frac{\partial^{2}f}{\partial{Q}_{i,j}\partial{Q}_{k,l}}{(Q\Delta)}_{ij}{(Q\Delta)}_{kl}
+ \frac{1}{2}\text{tr}((\nabla f(\matr{Q}))^{\T}\matr{Q}\matr{\Delta}^2 + \matr{\Delta}(\nabla f(\matr{Q}))^{\T}\matr{Q}\matr{\Delta})\notag\\
& = 6\sum\limits_{i,j,k}(3\tenselem{V}_{ijj}\tenselem{V}_{kjj}+2\tenselem{W}_{jjj}\tenselem{U}_{ikj}){(Q\Delta)}_{ij}{(Q\Delta)}_{kj}
+ \text{tr}((\nabla f(\matr{Q}))^{\T}\matr{Q}\matr{\Delta}^2) \notag\\
& = 6\sum\limits_{i,j,k}(3\tenselem{V}_{ijj}\tenselem{V}_{kjj}+2\tenselem{W}_{jjj}\tenselem{U}_{ikj}){(Q\Delta)}_{ij}{(Q\Delta)}_{kj}
-6\sum\limits_{i,j,k,l}{Q}_{ik}\tenselem{V}_{ijj}\tenselem{W}_{jjj}\Delta_{kl}\Delta_{jl}\label{eq-hessian-Q}\\
& = 6\sum\limits_{j}(3\tenselem{X}_{jjj}^2+2\tenselem{Y}_{jjj}\tenselem{W}_{jjj} + \tenselem{Z}_{jjj}\tenselem{W}_{jjj}).\notag \qedhere
\end{align}
\end{proof}

\begin{corollary}\label{coro-hessian}
Let $\matr{Q} =  \matr{I}_{n}$ in Lemma~\ref{lemma-Riemannian-Hessian}.
The tangent vector space $\text{T}_{\matr{I}_{n}}\ON{n}$  contains the skew symmetric matrices $\matr{\Delta}$.
It follows by \eqref{eq-hessian-Q} that 
\begin{align*}
\text{Hess}f(\matr{I}_{n})(\matr{\Delta},\matr{\Delta})
&= 6\sum\limits_{i,j,k}(3\tenselem{A}_{ijj}\tenselem{A}_{kjj}+2\tenselem{A}_{jjj}\tenselem{A}_{ikj})\Delta_{ij}\Delta_{kj}
-6\sum\limits_{i,j,k}\tenselem{A}_{kii}\tenselem{A}_{iii}\Delta_{ij}\Delta_{kj}\\
& = 6\sum\limits_{i,j}(3\tenselem{A}_{ijj}^{2}+2\tenselem{A}_{jjj}\tenselem{A}_{iij}
-\tenselem{A}_{iii}^{2})\Delta_{ij}^{2}\\
&+6\sum\limits_{i,j,k,k\neq i}(3\tenselem{A}_{ijj}\tenselem{A}_{kjj}+2\tenselem{A}_{jjj}\tenselem{A}_{ikj}
-\tenselem{A}_{kii}\tenselem{A}_{iii})\Delta_{ij}\Delta_{kj}
\end{align*}
\end{corollary}

\begin{remark}\label{theorem-JD-LMD-n}
Let $\tens{A}\in\set{JD}_n$.\\
(i) If
$\text{Hess}f(\matr{I}_{n})(\Delta,\Delta)< 0$
for any $\Delta\in \text{T}_{\matr{I}_{n}}\ON{n} \setminus \{\matr{0}\}$,
then $\tens{A}\in\set{LMD}_n$.\\
(ii) If $\tens{A}\in\set{LMD}_n$,
then
$\text{Hess}f(\matr{I}_{n})(\Delta,\Delta)\le 0$
for any $\Delta\in \text{T}_{\matr{I}_{n}}\ON{n}$.
\end{remark}

\subsection{Euclidean Hessian matrix for $\set{S}_3$}
Note that $\set{LMD}_n\subseteq\set{JD}_n$  and $\set{LMD}_n$  is corresponding to the local maximum point of \eqref{cost-function-f-q}. In this subsection, based on Corollary~\ref{coro-hessian},  we show how to determine whether $\tens{A}\in\set{JD}_3$ is locally maximally diagonal or not.

\begin{definition}\label{EH-Matrix}
Let $\tens{A}\in\set{JD}_3$.
Let $\gamma_{12},\gamma_{13}$ and $\gamma_{23}$ be the stationary diagonal ratios  introduced in Definition~\ref{re-defi-index}.
Denote by
\[
a = \tenselem{A}_{111},\quad b = \tenselem{A}_{222},\quad c = \tenselem{A}_{333} \quad \text{and}\quad g= \tenselem{A}_{123}. 
\]
We define the \emph{Euclidean Hessian matrix}  of $\tens{A}$ to be
$\matr{M}_{\tens{A}}\eqdef$
\begin{align*}
\begin{bmatrix}
(3\gamma_{12}^2+2\gamma_{12}-1)(a^2+b^2) & 2ga+(3\gamma_{12}\gamma_{13}-\gamma_{23})bc
& -2gb-(3\gamma_{23}\gamma_{12}-\gamma_{13})ca \\
2ga+(3\gamma_{12}\gamma_{13}-\gamma_{23})bc & (3\gamma_{13}^2+2\gamma_{13}-1)(c^2+a^2)
& 2gc+(3\gamma_{13}\gamma_{23}-\gamma_{12})ab \\
-2gb-(3\gamma_{23}\gamma_{12}-\gamma_{13})ca & 2gc+(3\gamma_{13}\gamma_{23}-\gamma_{12})ab
& (3\gamma_{23}^2+2\gamma_{23}-1)(b^2+c^2)
\end{bmatrix}.
\end{align*}
\end{definition}

\begin{theorem}\label{theorem-JD-LMD}
Let $\tens{A}\in\set{JD}_3$.
If $\matr{M}_{\tens{A}}$ is negative definite,
then $\tens{A}\in\set{LMD}_3$. 
If $\tens{A}\in\set{LMD}_3$, 
then $\matr{M}_{\tens{A}}$ is negative semidefinite.
\end{theorem}

\begin{proof}%
Let
\[
\Delta = \begin{bmatrix}
0 & u & v \\
-u & 0 & w \\
-v & -w & 0
\end{bmatrix}
\in \text{T}_{\matr{I}_{3}}\ON{3}.
\]
Define $\Phi(u,v,w) \eqdef \text{Hess}f(\matr{I}_3)(\Delta,\Delta)$. 
By Corollary~\ref{coro-hessian}, we have that
\begin{align}
\Phi(u,v,w)
&= 6[(3\gamma_{12}^2+2\gamma_{12}-1)(a^2+b^2)u^2+(3\gamma_{13}^2+2\gamma_{13}-1)(c^2+a^2)v^2\notag\\
&+(3\gamma_{23}^2+2\gamma_{23}-1)(b^2+c^2)w^2 +4g(auv+cvw-bwu)\notag\\
&+ 6(\gamma_{12}\gamma_{13}bcuv+\gamma_{13}\gamma_{23}abvw-\gamma_{23}\gamma_{12}cawu)
-2(\gamma_{23}bcuv+\gamma_{12}abvw-\gamma_{13}cawu)]\notag\\
&=6\,\xi^{\T} \matr{M}_{\!\tens{A}}\,\xi,\label{eq-Phi}
\end{align}
 where $\xi=(u,v,w)^{\T}$. By Remark~\ref{theorem-JD-LMD-n}, the proof is complete.  %
\end{proof}

\begin{example}\label{LC2-charact-01}
Let $\tens{A}\in\set{JD}_3$ be such  that $\tenselem{A}_{111} = \tenselem{A}_{222} = \tenselem{A}_{333} = 1$,
$\gamma_{12} = \gamma_{13} = \gamma_{23} = \gamma$.\\
(i) Then
\begin{align*}
\Phi(u,v,w) = 12(u^2+v^2+w^2)\big[(3\gamma^2+2\gamma-1) - (3\gamma^2-\gamma+2\tenselem{A}_{123})\frac{uw-uv-vw}{u^2+v^2+w^2}\big]
\end{align*}
for any $(u,v,w)\in\RR^3\backslash\{(0,0,0)\}$ by \eqref{eq-Phi}.
Note that
\[\frac{uw-uv-vw}{u^2+v^2+w^2}\in[-1/2,1].\]
Since
$3\gamma^2+2\gamma-1\le0$
by Theorem~\ref{theorem-gamma-charac},
it follows that
 $\tens{A}\in\set{LMD}_3$ 
if
\begin{align*}
\frac{3}{2}\gamma - \frac{1}{2} < \tenselem{A}_{123} < -\frac{9}{2}\gamma^2-\frac{3}{2}\gamma+1.
\end{align*}
Moreover,
we have that  $\tens{A}\notin\set{LMD}_3$, 
if
\begin{align*}
\tenselem{A}_{123}<\frac{3}{2}\gamma - \frac{1}{2}\quad\text{or}\quad\tenselem{A}_{123}> -\frac{9}{2}\gamma^2-\frac{3}{2}\gamma+1.
\end{align*}
(ii) If $\gamma = 0$,
then
\begin{align*}
\Phi(u,v,w) = -12(u^2+v^2+w^2)[1+2\tenselem{A}_{123}\frac{uw-uv-vw}{u^2+v^2+w^2}].
\end{align*}
for any $(u,v,w)\in\RR^3\backslash\{(0,0,0)\}$.
Note that
\[\frac{uw-uv-vw}{u^2+v^2+w^2}\in[-1/2,1].\]
It follows that $\tens{A}\in\set{LMD}_3$ 
if $\tenselem{A}_{123}\in(-1/2,1)$.
Moreover,
if $\tenselem{A}_{123}\notin[-1/2,1]$,
then  $\tens{A}\notin\set{LMD}_3$. 
\end{example}

\begin{example}
Let $\tens{A}\in\set{S}_n$ with $n>3$.
Suppose that
\[
\tens{A}^{(i,j,k)}\in\set{LMD}_3
\]
for any $1\le i<j<k\le n$.
It may be interesting to  wonder   whether it holds that
\[
\tens{A}\in\set{LMD}_n.
\]
In fact,
the answer is negative.
Let $\tens{A}\in\set{PD}_4$ with
\begin{align*}
\tenselem{A}_{ijk} =
\begin{cases}
1, &  i=j=k, \\
3/4, &  i\neq j\neq k, \\
0, &  \text{otherwise}. \\
\end{cases}
\end{align*}
By Example~\ref{LC2-charact-01}~(ii),
we see that $\tens{A}^{(i,j,k)}\in\set{LMD}_3$ for any $1\le i<j<k\le n$.
Let
\[
\Delta_{*} = \begin{bmatrix}
0 & 1 & 1 & 1\\
-1 & 0 & 0 & 0\\
-1 & 0 & 0 & 0\\
-1 & 0 & 0 & 0
\end{bmatrix}
\in \text{T}_{\matr{I}_{4}}\ON{4}.
\]
By  Corollary~\ref{coro-hessian},
we get that
$\text{Hess}f(\matr{I}_{4})(\Delta_{*},\Delta_{*}) =  18>0$. 
It follows that $\tens{A}\notin\set{LMD}_4$.
\end{example}

\section{Orbit of generally maximally diagonal tensors}\label{sec:GMD}
\subsection{Equivalent problem formulations}
In this subsection,
we first prove that the statement $\set{O}(\set{GMD}_n) = \set{S}_{n}$  is equivalent to several other optimization problems
in Proposition~\ref{lemma-equiva-statemnts}.
Then we give a positive answer to these equivalent problems when the dimension is 2 in Theorem~\ref{equal-dim-2}.
\begin{proposition}\label{lemma-equi-GMD}
Let $n\geq2$.
Then
$\set{O}(\set{GMD}_n) = \set{S}_{n}$ if and only if $\set{GMD}_n = \set{MD}_n$.
\end{proposition}
\begin{proof}
We only have to prove that $\set{GMD}_n = \set{MD}_n$
if $\set{O}(\set{GMD}_n) = \set{S}_{n}$.
In fact,
if $\tens{A}\in\set{MD}_{n}$,
there exists $\matr{Q}_{*}$ such that 
\[
\tens{A} \contr{1} \matr{Q}_{*}^{\T}\contr{2} \matr{Q}_{*}^{\T}\contr{3} \matr{Q}_{*}^{\T} \in \set{GMD}_n.
\]
Let $f$ be as in \eqref{cost-function-f-q} and $\mathcal{F}$ be as in \eqref{eq-defi-f-pqr}.
It follows that
\[
f(\matr{I}_{n}) \geq f(\matr{Q}_{*})=\max\limits_{\matr{P},\matr{Q},\matr{R}\in\SON{n}}\mathcal{F}(\matr{P},\matr{Q},\matr{R})\geq f(\matr{I}_{n}).
\]
Then we have that $\tens{A}\in\set{GMD}_n$.
\end{proof}

\begin{proposition}\label{lemma-equiva-statemnts}
Denote
\[
 \set{GO}(\set{D}_n)  \eqdef
\{\tens{A}\contr{1}\matr{P}^{\T}\contr{2}\matr{Q}^{\T}\contr{3}\matr{R}^{\T},
\tens{A}\in\set{D}_n, \matr{P},\matr{Q},\matr{R}\in\ON{n}\}.
\]
The following statements are equivalent.

\noindent(i) $\set{GMD}_n = \set{MD}_n$. 

\noindent(ii) For any $\tens{A}\in\set{S}_n$,
\begin{equation*}\label{eq-max-equal}
\max\limits_{\matr{Q}\in\SON{n}}f(\matr{Q}) =
\max\limits_{\matr{P},\matr{Q},\matr{R}\in\SON{n}}\mathcal{F}(\matr{P},\matr{Q},\matr{R}),
\end{equation*}
 where $f$  is  as in \eqref{cost-function-f-q} and $\mathcal{F}$  is  as in \eqref{eq-defi-f-pqr}. 

\noindent(iii) For any $\tens{A}\in\set{S}_n$,
it holds that 
\begin{equation}\label{eq-low-rank-equal}
\min_{\substack{\vect{u}_i\perp \vect{u}_j,\forall i\neq j,\\\mu_k\in\RR}}\|\tens{A}-\sum_{k=1}^{n}\mu_k \, \vect{u}_k\otimes \vect{u}_k\otimes \vect{u}_k\|
=\min_{\substack{\vect{x}_i\perp \vect{x}_j,\vect{y}_i\perp \vect{y}_j,\\ \vect{z}_i\perp \vect{z}_j,\forall i\neq j,\lambda_k\in\RR}}
\|\tens{A}-\sum_{k=1}^{n}\lambda_k \, \vect{x}_k\otimes \vect{y}_k\otimes \vect{z}_k\|.
\end{equation} 
(iv) For any $\tens{A}\in\set{S}_n$ and for the Euclidean distance $d$, it holds that 
\begin{equation*}
 d(\tens{A},\set{O}(\set{D}_n)) = d(\tens{A},\set{GO}(\set{D}_n)). 
\end{equation*}
(v) Let $\tens{A}\in\set{S}_n$.
The best rank-$n$ orthogonal approximation can always be chosen to be symmetric,
that is,
there exist $\mu_k\in\RR$ and orthonormal basis $\{\vect{u}_k,1\le k\le n\}$ such that
\begin{equation*}
\|\tens{A}-\sum_{k=1}^{n}\mu_k \, \vect{u}_k\otimes \vect{u}_k\otimes \vect{u}_k\|
=\min_{\substack{\vect{x}_i\perp \vect{x}_j,\vect{y}_i\perp \vect{y}_j,\\ \vect{z}_i\perp \vect{z}_j,\forall i\neq j,\lambda_k\in\RR}}
\|\tens{A}-\sum_{k=1}^{n}\lambda_k \, \vect{x}_k\otimes \vect{y}_k\otimes \vect{z}_k\|.
\end{equation*}
\end{proposition}

\begin{proof}
(i)$\Leftrightarrow$(ii).
Suppose that (i) holds and
$\matr{Q}_{*}=\arg\max\limits_{\matr{Q}\in\SON{n}}f(\matr{Q}).$
Let
\[
\tens{W}_{*} = \tens{A} \contr{1} \matr{Q}_{*}^{\T}\contr{2} \matr{Q}_{*}^{\T}\contr{3} \matr{Q}_{*}^{\T}.
\]
Then $\tens{W}_{*}\in\set{MD}_n$  and thus $\tens{W}_{*}\in\set{GMD}_n$. 
It follows that
\[
f(\matr{Q}_{*})=\max\limits_{\matr{P},\matr{Q},\matr{R}\in\SON{n}}\mathcal{F}(\matr{P},\matr{Q},\matr{R}).
\]
If (ii) holds and  $\tens{A}\in\set{MD}_n$, 
then
$\matr{I}_{n}=\arg\max\limits_{\matr{Q}\in\SON{n}}f(\matr{Q})$
and thus
\[
(\matr{I}_{n},\matr{I}_{n},\matr{I}_{n})=\arg\max\limits_{\matr{P},\matr{Q},\matr{R}\in\SON{n}}\mathcal{F}(\matr{P},\matr{Q},\matr{R}),
\]
which implies that  $\tens{A}\in\set{GMD}_n$. \\
(ii)$\Leftrightarrow$(iii).
By \cite[Proposition 5.1]{chen2009tensor}, \cite[(5.6)]{chen2009tensor} and \cite[(5.23)]{chen2009tensor},
we get that 
\begin{align*}
&\max\limits_{\matr{Q}\in\SON{n}}f(\matr{Q}) = \|\tens{A}\|^2 -
\min_{\substack{\vect{u}_i\perp \vect{u}_j,\forall i\neq j,\\\mu_k\in\RR}}\|\tens{A}-\sum_{k=1}^{n}\mu_k \, \vect{u}_k\otimes \vect{u}_k\otimes \vect{u}_k\|^2,\\
& \max\limits_{\matr{P},\matr{Q},\matr{R}\in\SON{n}}\mathcal{F}(\matr{P},\matr{Q},\matr{R}) = \|\tens{A}\|^2 -
\min_{\substack{\vect{x}_i\perp \vect{x}_j,\vect{y}_i\perp \vect{y}_j,\\ \vect{z}_i\perp \vect{z}_j,\forall i\neq j,\lambda_k\in\RR}}
\|\tens{A}-\sum_{k=1}^{n}\lambda_k \, \vect{x}_k\otimes \vect{y}_k\otimes \vect{z}_k\|^2.
\end{align*} 
It follows that (ii)$\Leftrightarrow$(iii).\\
(iii)$\Leftrightarrow$(iv) is clear.\\
(iii)$\Leftrightarrow$(v).
Note that  $\set{O}(\set{D}_n)$ is closed.  There exist $\mu_k\in\RR$ and an orthonormal basis $\{\vect{u}^{*}_k,1\le k\le n\}$ such that 
\begin{equation*}
\|\tens{A}-\sum_{k=1}^{n}\mu_k \, \vect{u}_k\otimes \vect{u}_k\otimes \vect{u}_k\|
=\min_{\substack{\vect{v}_i\perp \vect{v}_j,\forall i\neq j,\\\mu_k\in\RR}}\|\tens{A}-\sum_{k=1}^{n}\mu_k \, \vect{v}_k\otimes \vect{v}_k\otimes \vect{v}_k\|. \qedhere
\end{equation*} 
\end{proof}
\begin{theorem}\label{equal-dim-2}
 It holds that $\set{MD}_2 = \set{GMD}_2.$ 
\end{theorem}
\begin{proof}
We only need to prove that $\tens{A} \in \set{GMD}_2$ if $\tens{A} \in \set{MD}_2$.
Let
\[\tens{W}=\tens{A}\contr{1}{\matr{P}}^{\T}\contr{2}{\matr{Q}}^{\T}\contr{3}{\matr{R}}^{\T}\]
with $\matr{P},\matr{Q},\matr{R}\in\SON{2}$.
These rotations can be written as
\begin{align*}
\matr{P}= \frac{1}{\sqrt{1+x^2}}
\begin{bmatrix}
1 & -x \\
x & 1
\end{bmatrix},
\matr{Q}= \frac{1}{\sqrt{1+y^2}}
\begin{bmatrix}
1 & -y \\
y & 1
\end{bmatrix}\  \text{and}\
\matr{R}= \frac{1}{\sqrt{1+z^2}}
\begin{bmatrix}
1 & -z \\
z & 1
\end{bmatrix}
\end{align*}
for $x,y,z\in\RR$.
Define
\[\mathcal{F}(x,y,z) \eqdef \|\diag{\tens{W}}\|^2\]
as in \eqref{eq-defi-f-pqr}.
Denote
\[a=\tenselem{A}_{111},b=\tenselem{A}_{112},c=\tenselem{A}_{122}, d=\tenselem{A}_{222},\]
 and $\gamma=\gamma_{12}$ is  the stationary diagonal ratio in Definition~\ref{re-defi-index}.
Then $c=\gamma a$ and $b=\gamma d$ by definition. 
Moreover, $\gamma\in[-1,1/3]$ by Theorem~\ref{theorem-gamma-charac}.
It can be calculated that
\begin{align*}
\mathcal{F}(x,y,z)
&=a^2+d^2 + \frac{(a^2+d^2)(\gamma+1)}{(1+x^2)(1+y^2)(1+z^2)} {\sigma(x,y,z),}
\end{align*}
 where 
\begin{align*}
\sigma(x,y,z) &\eqdef (\gamma-1)(x^2+y^2+z^2+x^2y^2+y^2z^2+z^2x^2) \\
&+ 2\gamma(x^2yz+xy^2z+xyz^2+xy+yz+zx).
\end{align*}
Note that $\mathcal{F}(0,0,0)=a^2+d^2.$
We only need to prove that $\sigma(x,y,z) \le0$ for any $x,y,z\in\RR$.
If $\gamma\in[0,1/3]$,
then $\gamma-1\le -2\gamma$,
and thus
\begin{align*}
\sigma(x,y,z) &\le -2\gamma [(x^2+y^2+z^2+x^2y^2+y^2z^2+z^2x^2) \\
& \phantom{\le -2\gamma[}-(x^2yz+xy^2z+xyz^2+xy+yz+zx)]\\
&= -\gamma [(x-y)^2+(y-z)^2+(z-x)^2+(xy-yz)^2+(yz-zx)^2+(zx-xy)^2]\le0.
\end{align*}
If $\gamma\in[-1,0)$,
then
\begin{align*}
\sigma(x,y,z) &= -(1+\gamma)(x^2+y^2+z^2+x^2y^2+y^2z^2+z^2x^2) \\
&+ \gamma[(x+y)^2+(y+z)^2+(z+x)^2+(xy+yz)^2+(yz+zx)^2+(zx+xy)^2]\le0.
\end{align*}
\end{proof}

\subsection{Symmetric tensors of dimension {$n > 2$}}
In this subsection,
we first present a counterexample to show that the equivalent problems in Proposition~\ref{lemma-equi-GMD} and Proposition~\ref{lemma-equiva-statemnts}
have a negative answer when $n>2$.
Then we prove a related result,
which can be seen as an orthogonal analogue
of the Comon's conjecture.

\subsubsection{A counterexample}

\begin{lemma}\label{lemma-before-counter}
Define
\begin{equation}\label{eq:rho_def}
\rho(\matr{Q}) \eqdef  {Q}^2_{11}{Q}^2_{12}{Q}^2_{13}
+{Q}^2_{21}{Q}^2_{22}{Q}^2_{23}+{Q}^2_{31}{Q}^2_{32}{Q}^2_{33} 
\end{equation}
{for $\matr{Q}\in\ON{3}$}.
We have that $\rho(\matr{Q}) < 1/12$ for any $\matr{Q}\in\SON{3}$.
\end{lemma}
The proof of Lemma~\ref{lemma-before-counter} can be found in \ref{proofs}.

\begin{example}\label{exam-counter}
Let $\tens{A}$ be as in Example~\ref{exam-not-equal-D-PD-1}.
Let $\mathcal{F}$ be as in \eqref{eq-defi-f-pqr}.
Suppose that
\begin{align*}
\matr{P}_{*}=
\begin{bmatrix}
1 & 0 & 0 \\
0 & 1  & 0\\
0 & 0  & 1
\end{bmatrix},
\matr{Q}_{*}=
\begin{bmatrix}
0 & 0 & 1 \\
1 & 0 & 0\\
0 & 1 & 0
\end{bmatrix},
\matr{R}_{*}=
\begin{bmatrix}
0 & 1 & 0 \\
0 & 0 & 1\\
1 & 0 & 0
\end{bmatrix}.
\end{align*}
Then
$\mathcal{F}(\matr{P}_{*},\matr{Q}_{*},\matr{R}_{*})=3.$
However, easy calculations show that 
\begin{equation*}
f(\matr{Q}) = 36\rho(\matr{Q}) < 3,
\end{equation*}
where the last inequality follows by  Lemma~\ref{lemma-before-counter}.
Thus, 
we see that
\[
f(\matr{Q}) <3=\mathcal{F}(\matr{P}_{*},\matr{Q}_{*},\matr{R}_{*})
\]
for any $\matr{Q}\in\ON{3}$.
It follows that
Proposition~\ref{lemma-equiva-statemnts}~(ii) has a negative answer when $n>2$.
Moreover,
we have $\set{O}(\set{GMD}_n) \subsetneqq \set{S}_{n}$  when $n>2$ by Proposition~\ref{lemma-equi-GMD}.
\end{example}

\begin{remark}
It was proved that the best rank-1 approximation of
any $\tens{A}\in\set{S}_n$ can always be chosen to be symmetric \cite{friedland2013best,zhang2012best}.
Example~\ref{exam-counter} provides a counterexample to Proposition~\ref{lemma-equiva-statemnts}~(v)  when $n>2$. 
It will be interesting to study whether the best rank-$p$ ($1<p<n$)
orthogonal approximation can be chosen to be symmetric when $n>2$, which can be seen as an orthogonal analogue of \cite[Conjecture 8.7]{friedland2016remarks}. 
\end{remark}

\subsubsection{An orthogonal analogue of Comon's conjecture}

Although Proposition~\ref{lemma-equiva-statemnts}~(iii) has a negative answer by Example~\ref{exam-counter} when $n>2$, we have the following result.

\begin{proposition}\label{pro-analogue-Comon}
Let $\tens{A}\in\set{S}_n$.
Then for any $p$ 
\[
\min_{\substack{\vect{x}_i\perp \vect{x}_j,\vect{y}_i\perp \vect{y}_j,\\ \vect{z}_i\perp \vect{z}_j,\forall i\neq j,\lambda_k\in\RR}}
\|\tens{A}-\sum_{k=1}^{p}\lambda_k \, \vect{x}_k\otimes \vect{y}_k\otimes \vect{z}_k\|=0
\]
 implies
\[
\min_{\substack{\vect{u}_i\perp \vect{u}_j,\forall i\neq j,\\\mu_k\in\RR}}\|\tens{A}-\sum_{k=1}^{p}\mu_k \, \vect{u}_k\otimes \vect{u}_k\otimes \vect{u}_k\|=0.
\]
\end{proposition}

\begin{proof}
Suppose that  $1 \le p \le n$  and 
\begin{equation}\label{eq-propo-proof}
\tens{A} = \sum_{k=1}^{p}\lambda_k \, \vect{x}_k\otimes \vect{y}_k\otimes \vect{z}_k,
\end{equation} 
where $\lambda_k\in\RR\setminus\{0\}$  and $\vect{x}_i\perp \vect{x}_j,\vect{y}_i\perp \vect{y}_j, \vect{z}_i\perp \vect{z}_j$ for any $i\neq j$.
We assume that  $\|\vect{x}_k\| = \|\vect{y}_k\| = \|\vect{z}_k\| = 1$  without loss of generality.
Note that $\tens{A}$ is symmetric.
Then
\begin{align*}
 \tens{A}\contr{3} \vect{z}_k = \lambda_k \, \vect{x}_k\otimes \vect{y}_k 
\end{align*}
is a symmetric matrix for any $1\leq k\leq p$.
It follows that $\vect{x}_k=\pm \vect{y}_k$.
In a similar way, we can prove that 
$\vect{y}_k=\pm \vect{z}_k$.
The proof is complete.
\end{proof}

\begin{corollary}\label{corollary-analogue-Comon}
Let $\tens{A}\in\set{S}_n$.
Then we have 
\[\text{d}(\tens{A},\set{GO}(\set{D}_n))=0\ \Rightarrow\ \text{d}(\tens{A},\set{O}(\set{D}_n))=0,\]
that is, 
\[\set{O}(\set{D}_n)
=\set{S}_n\cap\set{GO}(\set{D}_n).\] 
\end{corollary}

\begin{remark}
(i)
Proposition~\ref{pro-analogue-Comon} can be seen as an orthogonal analogue
of the Comon's conjecture \cite{Comon08:symmetric,friedland2016remarks,zhang2016comon},
which conjectured that rank and symmetric rank of a symmetric tensor are equal, 
that is,
\begin{equation*}
\min_{\substack{ \vect{x}_k,\vect{y}_k,\vect{z}_k \in \RR^n ,\\\lambda_k\in\RR}}
\|\tens{A}-\sum_{k=1}^{p}\lambda_k \, \vect{x}_k\otimes \vect{y}_k\otimes \vect{z}_k\| = 0
\Rightarrow
\min_{ \vect{u}_k\in\RR^n,  \mu_k\in\RR}
\|\tens{A}-\sum_{k=1}^{p}\mu_k \, \vect{u}_k\otimes \vect{u}_k\otimes \vect{u}_k\| = 0
\end{equation*}
for any $\tens{A}\in\set{S}_n$ and $p\in\NN$ minimal. 

\noindent(ii) An alternative proof of Corollary~\ref{corollary-analogue-Comon}
can be found in \cite[Proposition 32]{boralevi2015orthogonal}.
\end{remark}

\section{Convergence results for cyclic Jacobi algorithm}\label{sec:convergence}
\subsection{Cyclic Jacobi algorithm description}
In this subsection,
we recall the cyclic Jacobi algorithm (also called the Jacobi CoM2 algorithm) given in \cite{Como94:sp,Como10:book}, which is a special case of Algorithm~\ref{alg:jacobi-type}. 
\begin{algorithm}\label{alg:jacobi-CoM}
{\bf Input:} $\tens{A}\in\set{S}_n$ and $\matr{Q}_{0}=\matr{I}_{n}$.\\
{\bf Output:} a sequence of iterations $\{\matr{Q}_k:k\in\NN\}$. 
\begin{itemize}
\item {\bf For} $k=1,2,\ldots$ until a stopping criterion is satisfied do: 
  \begin{itemize}
  \item\quad Choose the pair $(i_k,j_k)$ according to the following cyclic-by-row rule
  \begin{equation*}\label{equation-Jacobi-C}
  \begin{split}
&(1,2) \to (1,3) \to \cdots \to (1,n) \to \\
& (2,3) \to \cdots \to (2,n) \to \\
& \cdots  \to \\
& (n-1,n)  \to \\
&(1,2) \to (1,3) \to \cdots.
  \end{split}
  \end{equation*}
  \item\quad Compute the angle $\theta_{k}^{*}$ that maximizes  the function $h_k(\theta)$ defined in  \eqref{eq-def-h-k}.
  \item\quad  Update $\matr{Q}_k = \matr{Q}_{k-1} \Gmat{i_k}{j_k}{\theta_{k}^{*}}$.
  \end{itemize}
\item {\bf End for}
\end{itemize}
\end{algorithm}

\subsection{Derivatives and relations between them}
In this subsection,
we present some basic properties of Algorithm~\ref{alg:jacobi-CoM}.
More details can be found in \cite{Como94:sp,LUC2017globally}.
We first give a definition.

Take the $k$-th iteration with pair $(i_k,j_k)$ in Algorithm~\ref{alg:jacobi-CoM}.
Let
\[\tens{W}^{(k-1)}=\tens{A} \contr{1}\matr{Q}^{\T}_{k-1}\contr{2}\matr{Q}^{\T}_{k-1}\contr{3}\matr{Q}^{\T}_{k-1}.\]
By \eqref{eq-def-h-k},
we have that
\begin{equation}\label{eq:definition-h-W}
h_k(\theta) =  \|\diag{\tens{W}^{(k-1)}\contr{1}(\Gmat{i_k}{j_k}{\theta})^{\T}
\contr{2}(\Gmat{i_k}{j_k}{\theta})^{\T}\contr{3}(\Gmat{i_k}{j_k}{\theta})^{\T}}\|^2.
\end{equation}
Let $x=\tan(\theta)$, and define
\[
\tau_k:\RR\rightarrow\RR\quad\text{by}\quad\tau_k(x) \eqdef h_k(\arctan(x)).
\]

In the rest of this subsection, with some abuse of notation, we use a shorthand notation   
${{d}_{k}}={d}_{i_k,j_k}(\tens{W}^{(k-1)})$
and ${\omega_{k}}=\omega_{i_k,j_k}(\tens{W}^{(k-1)})$.
 It  can be calculated that \cite[Lemma 5.8]{LUC2017globally}
\begin{align}
&\tau_k(x)-\tau_k(0)=\frac{3}{(1+x^2)^2}
(2{{d}_{k}}(x-x^3)-{\omega_{k}}x^2),\label{equa-increase-h}\\
&\tau'_k(x) = \frac{6}{(1+x^2)^3}({{d}_{k}}(1-6x^2+x^4)-{\omega_{k}}(x-x^3)),\label{equa-derivative-h}\\
&\tau''_k(x) = \frac{6}{(1+x^2)^4}[2{{d}_{k}}(-9x+14x^3-x^5)-{\omega_{k}}(1-8x^2+3x^4)].\notag
\end{align}

\begin{remark}
Denote by $x_{k}^{*}=\tan(\theta_{k}^{*})$ the optimal point of $\tau_k(x)$.
Note that $\tau'_k(x_{k}^{*})=0$.
It follows by \eqref{equa-derivative-h} that
\begin{equation}\label{eq-derivative-0}
{{d}_{k}}(1-6{x_{k}^{*}}^2+{x_{k}^{*}}^4)-{\omega_{k}}(x_{k}^{*}-{x_{k}^{*}}^3)=0.
\end{equation}
(i) If $x_{k}^{*}-{x_{k}^{*}}^3\neq 0$,
we get that
\begin{equation*}\label{increa-01}
{\omega_{k}} = \frac{(1-6{x_{k}^{*}}^2+{x_{k}^{*}}^4)}{x_{k}^{*}(1-{x_{k}^{*}}^2)}{{d}_{k}},\ \
\text{and thus}\ \
\tau_k(x_{k}^{*}) - \tau_k(0) = \frac{3x_{k}^{*}}{(1-{x_{k}^{*}}^2)}{{d}_{k}}.
\end{equation*}
(ii) If $1-6{x_{k}^{*}}^2+{x_{k}^{*}}^4\neq0$,
we get that
\begin{equation*}
{{d}_{k}} = \frac{x_{k}^{*}(1-{x_{k}^{*}}^2)}{(1-6{x_{k}^{*}}^2+{x_{k}^{*}}^4)}{\omega_{k}},
\end{equation*}
and thus
\begin{equation}\label{eq-doule-derivative}
\begin{split}
&\tau''_k(x_{k}^{*}) = \frac{-6{\omega_{k}}}{(1-6{x_{k}^{*}}^2+{x_{k}^{*}}^4)},\\
&\tau_k(x_{k}^{*}) - \tau_k(0) = \frac{3{x_{k}^{*}}^2}{(1-6{x_{k}^{*}}^2+{x_{k}^{*}}^4)}{\omega_{k}}.
\end{split}
\end{equation}
\end{remark}

\subsection{Convergence properties}
In this subsection
we prove some results on the convergence properties of Algorithm~\ref{alg:jacobi-CoM}.

\begin{proposition}\label{theorem-convergence-stationary}
Suppose that $\tens{A}\in\set{S}_n$
and $\{\matr{Q}_k:k\in\NN\}\subseteq\SON{n}$ are the iterations
of Algorithm~\ref{alg:jacobi-CoM}.
If $\matr{Q}_k\rightarrow\matr{Q}_{*}$ and
\[\tens{W}^{*} = \tens{A}\contr{1}\matr{Q}_{*}^{\T}\contr{2}\matr{Q}_{*}^{\T}\contr{3}\matr{Q}_{*}^{\T},\]
then $d_{i,j}(\tens{W}^{*})=0$
and $\omega_{i,j}(\tens{W}^{*})\geq0$ for any $1\le i<j\le n$.
\end{proposition}

\begin{proof}
Fix any $1\le i_{*}<j_{*}\le n$.
We choose a subsequence $\LL\subseteq\NN$ such that
\[
(i_{\ell+1},j_{\ell+1})=(i_{*},j_{*})
\]
for any $\ell\in\LL$. 
It follows by $\matr{Q}_k\rightarrow\matr{Q}_{*}$ that
$x^{*}_{\ell+1}\rightarrow0$ when $\ell\in\LL$ tends to infinity.
Then we get that
$d_{i_{*},j_{*}}(\tens{W}^{(\ell)})\rightarrow0$
by \eqref{eq-derivative-0}.
Note that $\tau''_{\ell+1} (x^{*}_{\ell+1})\le0$ for any $\ell\in\LL$. 
By \eqref{eq-doule-derivative},
we have that $\omega_{i_{*},j_{*}}(\tens{W}^{(\ell)})\geq0$ when $\ell\in\LL$ is large enough.
Since $\matr{Q}_{\ell}\rightarrow\matr{Q}_{*}$,
the result follows from
continuity of the function
\begin{equation}\label{eq:multilinear_function-1}
\matr{Q} \mapsto  \tens{A}\contr{1}\matr{Q}^{\T}\contr{2} \matr{Q}^{\T}\contr{3} \matr{Q}^{\T}. \qedhere
\end{equation}
\end{proof}

\begin{proposition}\label{theorem-local-convergence-CoM}
Let $\tens{A}\in\set{S}_n$
and $\{\matr{Q}_k:k\in\NN\}\subseteq\SON{n}$ be the iterations
of Algorithm~\ref{alg:jacobi-CoM}.
Suppose that there are a finite number of accumulation points of $\{\matr{Q}_k:k\in\NN\}$.\\
(i) Let $\matr{Q}_{*}\in\SON{n}$ be any accumulation point and
\[\tens{W}^{*} = \tens{A}\contr{1}\matr{Q}_{*}^{\T}\contr{2}\matr{Q}_{*}^{\T}\contr{3}\matr{Q}_{*}^{\T}.\]
Then there exists $1\le i_{*}<j_{*}\le n$ such that
 $d_{i_{*},j_{*}}(\tens{W}^{*}) = \omega_{i_{*},j_{*}}(\tens{W}^{*})=0$.\\
(ii) For any $1\le i_{*}<j_{*}\le n$,
there exists an accumulation point $\matr{Q}_{*}\in\SON{n}$ such that
\[\tens{W}^{*} = \tens{A}\contr{1}\matr{Q}_{*}^{\T}\contr{2}\matr{Q}_{*}^{\T}\contr{3}\matr{Q}_{*}^{\T}\]
satisfies
 $d_{i_{*},j_{*}}(\tens{W}^{*}) = 0$ and $\omega_{i_{*},j_{*}}(\tens{W}^{*})\geq0$.\\
(iii) We have that the directional derivative (\ref{eq-d-omega}) tends to zero:
\[
h_{k+1}'(0)= 6 d_{i_{k+1},j_{k+1}}(\tens{W}^{(k)})\rightarrow0.
\]
\end{proposition}
 The proof can be found in \ref{proofs}.

\begin{corollary}\label{corollary-convergence-stationary}
Suppose that $\tens{A}\in\set{S}_n$ and $\{\matr{Q}_k:k\in\NN\}\subseteq\SON{n}$ are the iterations
of Algorithm~\ref{alg:jacobi-CoM}.
Let $\matr{Q}_{*}\in\SON{n}$ be an accumulation point and
\[\tens{W}^{*} = \tens{A}\contr{1}\matr{Q}_{*}^{\T}\contr{2}\matr{Q}_{*}^{\T}\contr{3}\matr{Q}_{*}^{\T}.\]
If $\omega_{i,j}(\tens{W}^{*}) > 0$
for any $1\le i<j\le n$,
then either  $\matr{Q}_k\rightarrow\matr{Q}_{*}$,
or there exist an infinite number of accumulation points in the iterations.
\end{corollary}

\begin{remark}\label{remark-LC2-C3}
By Proposition~\ref{theorem-convergence-stationary} and Lemma~\ref{lemma-C3-h},
we see that if the iterations of Algorithm~\ref{alg:jacobi-CoM} converge to $\matr{Q}_{*}$,
then
\[
\tens{W}^{*} = \tens{A}\contr{1}\matr{Q}_{*}^{\T}\contr{2}\matr{Q}_{*}^{\T}\contr{3}\matr{Q}_{*}^{\T}
\]
satisfies $\tens{W}^{*}\in\set{JD}_n$; 
in particular, 
$\matr{Q}_{*}$ is a stationary point of \eqref{cost-function-f-q} by Remark~\ref{remark-stationary}.
 However, $\set{LMD}_3\varsubsetneqq \set{JD}_3$ by Example~\ref{LC2-charact-01}. 
It follows that Algorithm~\ref{alg:jacobi-CoM} may converge to a saddle point of \eqref{cost-function-f-q}.
\end{remark}

\section{Conclusions}
In this paper,
we studied several classes of third order approximately diagonal tensors,
which are closely related to Jacobi-type algorithms and the approximate diagonalization problem \eqref{pro-ortho-diagonal}.
We believe that these classes provide a better understanding of problem \eqref{pro-ortho-diagonal} and behavior of optimization algorithms; some examples in this paper can be used as test cases for the algorithms. 
There are some open questions left for future research, such as the global convergence of Algorithm~\ref{alg:jacobi-CoM} for third (or higher) order symmetric tensors.

\section*{Acknowledgements}
The authors would like to thank the referees for their valuable comments and suggestions. This work was supported by the ERC project ``DECODA'' no.320594, in the frame of the European program FP7/2007-2013.
The first author was partially supported by the National Natural Science Foundation of China (No.11601371).

\appendix

\section{ Remaining proofs }\label{proofs}
\begin{proof}[Proof of Lemma~\ref{lemma-before-counter}]
 Since \eqref{eq:rho_def} is invariant with respect to changes of signs of the columns of $\matr{Q}$, it suffices to prove the statement for $\matr{Q}\in\SON{3}$. 

\noindent\textbf{Step 1.}
By \cite[p. 10]{murnaghan1962unitary},
any $\matr{Q}\in\SON{3}$ can be decomposed 	as
$\matr{Q}^{\T}=\matr{Q}_{1}(x)\matr{Q}_{2}(y)\matr{Q}_{3}(z)$,
where
\begin{align*}
&\matr{Q}_{1}(x) = \frac{1}{\sqrt{1+x^2}}
\begin{bmatrix}
\sqrt{1+x^2} & 0 & 0 \\
0 & 1  & -x\\
0 & x  & 1
\end{bmatrix},
\matr{Q}_{2}(y) = \frac{1}{\sqrt{1+y^2}}
\begin{bmatrix}
1 & -y & 0 \\
y & 1  & 0\\
0 & 0  &\sqrt{1+y^2}
\end{bmatrix},\\
&\matr{Q}_{3}(z) = \frac{1}{\sqrt{1+z^2}}
\begin{bmatrix}
1 & 0 & -z \\
0 & \sqrt{1+z^2}  & 0\\
z & 0  & 1
\end{bmatrix}
\end{align*}
for $x,y,z\in\RR$.
It can be calculated that
\begin{align*}
\rho(x,y,z) &\eqdef \rho(\matr{Q}) = \frac{1}{(1+x^2)^2(1+y^2)^3(1+z^2)^2}
[(y^4z^2 + y^2z^2)(x^4+1)\\
&+2\sqrt{y^2+1}y^3(z^3-z)(x^3-x)+(y^4z^4-4y^4z^2+y^4+y^2z^4+y^2+z^2)x^2].
\end{align*}
If $x=0$,
then
\begin{align*}
\rho(0,y,z) &= \frac{y^2z^2}{(1+y^2)^2(1+z^2)^2}\le \frac{1}{16}< \frac{1}{12}.
\end{align*}
The similar result holds if $z=0$.
Therefore,
we only need to prove that $\rho(x,y,z)<1/12$ in the case that $xz\neq0$.\\
\textbf{Step 2.}
Let
\[u = x-\frac{1}{x}\quad\text{and}\quad v = z-\frac{1}{z}.\]
We define
\begin{align*}
\Phi(u,v,y) \eqdef \rho(x,y,z) &= \frac{1}{(u^2+4)(v^2+4)(1+y^2)^3}
[(y^4+y^2)u^2\\
&+2\sqrt{y^2+1}y^3vu+(y^4+y^2)(v^2+4)+1-4y^4].
\end{align*}
Let $(u_{*},v_{*},y_{*})$ be the maximal point.
If $y_{*}=0$,
then
\begin{align*}
\Phi(u_{*},v_{*},0) = \frac{1}{(u_{*}^2+4)(v_{*}^2+4)}\le \frac{1}{16}< \frac{1}{12}.
\end{align*}
Now we prove that $u_{*}^2 = v_{*}^2$ if $y_{*}\neq0$.
Assume that $u_{*}^2 \neq v_{*}^2$.
By
\[\frac{\partial\Phi}{\partial u}(u_{*},v_{*},y_{*}) = \frac{\partial\Phi}{\partial v}(u_{*},v_{*},y_{*}) = 0,\]
we get that
\begin{align}
&u_{*}[v_{*}^2(y_{*}^4+y_{*}^2)+1-4y_{*}^4] = -y_{*}^3\sqrt{1+y_{*}^2}v_{*}(u_{*}^2-4)\label{eq-derivative-01},\\
&v_{*}[u_{*}^2(y_{*}^4+y_{*}^2)+1-4y_{*}^4] = -y_{*}^3\sqrt{1+y_{*}^2}u_{*}(v_{*}^2-4)\label{eq-derivative-02}.
\end{align}
If $u_{*}=0$,
then $v_{*}=0$ by \eqref{eq-derivative-01},
which implies that $u_{*}^2 = v_{*}^2$.
Otherwise,
if $u_{*}\neq0$,
then $v_{*}\neq0$.
It follows that
\begin{equation*}
u_{*}^2(v_{*}^2-4)[v_{*}^2(y_{*}^4+y_{*}^2)+1-4y_{*}^4] = v_{*}^2(u_{*}^2-4)[u_{*}^2(y_{*}^4+y_{*}^2)+1-4y_{*}^4]
\end{equation*}
by \eqref{eq-derivative-01} and \eqref{eq-derivative-02}.
It can be calculated that
\begin{equation*}\label{eq-after-derivative}
(y_{*}^4+y_{*}^2)u_{*}^2v_{*}^2(u_{*}^2-v_{*}^2) = -4(1-4y_{*}^4)(u_{*}^2-v_{*}^2).
\end{equation*}
By the assumption that $u_{*}^2 \neq v_{*}^2$,
we have
\begin{equation}\label{eq-contradiction-01}
u_{*}^2v_{*}^2 = \frac{-4(1-4y_{*}^4)}{y_{*}^4+y_{*}^2}.
\end{equation}
Moreover,
by \eqref{eq-derivative-01} and \eqref{eq-derivative-02},
we also get
\begin{align*}
(1-4y_{*}^4)(u_{*}^2-v_{*}^2) = -y_{*}^3\sqrt{1+y_{*}^2}u_{*}v_{*}(u_{*}^2-v_{*}^2),
\end{align*}
which implies that
\begin{align}\label{eq-contradiction-02}
u_{*}v_{*} = \frac{1-4y_{*}^4}{-y_{*}^3\sqrt{1+y_{*}^2}}.
\end{align}
By \eqref{eq-contradiction-01} and \eqref{eq-contradiction-02},
we get that $1-4y_{*}^4=0$.
It follows that $u_{*}v_{*}=0$ by \eqref{eq-contradiction-01},
which contradicts the assumption that $u_{*}\neq0$.
Therefore,
we prove that $u^2_{*}=v^2_{*}$.

\noindent\textbf{Step 3.}
Now we define
\begin{align}
&\psi(u,y) \eqdef \Phi(u,\pm u,y) = \frac{2(y^4+y^2\pm\sqrt{y^2+1}y^3)u^2+4y^2+1}{(u^2+4)^2(1+y^2)^3},\label{eq:def_psi}\\
&\varphi(u,y) \eqdef \frac{4y^2u^2(1+y^2)+4y^2+1}{(u^2+4)^2(1+y^2)^3}.\label{eq:def_varphi}
\end{align}
Note that
\begin{align*}
\varphi(u,y) = \frac{2(y^4+y^2+\sqrt{y^2+1}\sqrt{y^2+1}y^2)u^2+4y^2+1}{(u^2+4)^2(1+y^2)^3} \geq \psi(u,y)
\end{align*}
for any $u,y\in\RR$.
It is enough to prove that $\varphi(u,y)<1/12$ for any $u,y\in\RR$.
Let $(u_{**},y_{**})$ be the maximal point of $\varphi(u,y)$.
By
\[\frac{\partial\varphi}{\partial u}(u_{**},y_{**})=\frac{\partial\varphi}{\partial y}(u_{**},y_{**})=0,\]
we have that
\begin{align}
&y_{**}(4u_{**}^2y_{**}^4-4u_{**}^2+8y_{**}^2-1)= 0,\label{example-01}\\
&u_{**}(2u_{**}^2y_{**}^4+2u_{**}^2y_{**}^2-8y_{**}^4-4y_{**}^2+1)= 0\label{example-02}.
\end{align}
If $y_{**} = 0$,
then $u_{**} = 0$ by \eqref{example-02}.
If $u_{**} = 0$,
then $y_{**} = 0$ or $y_{**}^{2} = 1/8$ by \eqref{example-01}.
It is easy to check that $\varphi(u_{**},y_{**}) < 1/12$ in all these cases.

Now we assume that $y_{**} \neq 0$ and $u_{**} \neq 0$.
Then \eqref{example-02}  can be rewritten as
\begin{equation}\label{eq:u2_via_y2}
u_{**}^2 =\frac{8y_{**}^4+4y_{**}^2-1}{2y_{**}^2 (y^2_{**}+1)}.
\end{equation}
By substituting \eqref{eq:u2_via_y2} into \eqref{eq:def_varphi}, we get that
\[
\varphi(u_{**},y_{**}) = \frac{4y_{**}^4}{16y_{**}^6 + 18y_{**}^4+11y_{**}^2-1}.
\]
Next, we substitute \eqref{eq:u2_via_y2} into   \eqref{example-01}, and get that $y^{**}$ should satisfy 
\begin{align*}
(1-8y_{**}^2)(2y_{**}^4+2y_{**}^2) = (4y_{**}^4-4)(8y_{**}^4+4y_{**}^2-1).
\end{align*}
After division by $(y_{**}^2+1)$, we have
\begin{equation}\label{eq:3d_polynomial}
16y_{**}^6-11y_{**}^2+2=0,
\end{equation}
which is a $3$rd  degree polynomial equation in $y_{**}^2$;
there are two positive solutions of \eqref{eq:3d_polynomial} given by positive roots of the polynomial,  i.e.,  
$y_{**}^2\approx0.7162$ or $y_{**}^2\approx 0.1921$.
Taking into account \eqref{eq:3d_polynomial}, we have have that
\[
\varphi(u_{**},y_{**}) = \frac{4y_{**}^4}{28y_{**}^4+22y_{**}^2-1},
\]
hence $\varphi(u_{**},y_{**})\approx 0.076 <1/12$ or $\varphi(u_{**},y_{**})\approx 0.065 <1/12$ in these two cases.

\noindent\textbf{Step 4.}
Finally,
we have that
\[\rho(x,y,z)\le\max\limits_{u,v,y\in\RR}\Phi(u,v,y)=\max\limits_{u,y\in\RR}\psi(u,y)\le\varphi(u_{**},y_{**})<\frac{1}{12}\]
for any $x,y,z\in\RR$,
which completes the proof.
\end{proof}

\begin{proof}[Proof of Proposition~\ref{theorem-local-convergence-CoM}]
Since there are a finite number of accumulation points,
there exists $\delta>0$ such that the $\delta$-neighborhoods of these accumulation points have positive distance to each other.\\
\noindent(i) Let $\matr{Q}_{*}$ be any accumulation point.
Let $\LL\subseteq\NN$ be a subsequence such  the subsequence\footnote{We use a simplified notation for subsequences in order to avoid multilevel indices. }  $\{\matr{Q}_{\ell},\ell\in\LL\}$ is located in the $\delta$-neighborhood $\mathcal{N}(\matr{Q}_{*},\delta)$ and  $\matr{Q}_\ell\rightarrow\matr{Q}_{*}$ when $\ell\in\LL$ tends to infinity. 
Note that $\matr{Q}_{*}$ is not the unique accumulation point.
There exists a pair $(i_{*},j_{*})$ such that it appears for an infinite number of times in the sequence of pairs 
\[
\{(i_{\ell+1},j_{\ell+1}),\matr{Q}_{\ell+1}\notin\mathcal{N}(\matr{Q}_{*},\delta), \ell \in \LL\}.
\]
Now we construct a subsequence  $\PP\subseteq\LL$ such that
\[
(i_{p+1},j_{p+1})=(i_{*},j_{*})\quad \text{and}\quad\matr{Q}_{p+1}\notin\mathcal{N}(\matr{Q}_{*},\delta)
\]
for any $p\in\PP$. 
Note that the $\delta$-neighborhoods of different accumulation points have positive distance to each other.
There exists $\sigma>0$ such that $|x^{*}_{p+1}|>\sigma$ for any $p\in\PP$. 
Note that $|x^{*}_{p+1}|\leq 1$ for any $p\in\PP$. 
There exists $\zeta_0\in[-1,1]$ such that $\sigma\le|\zeta_0|\le1$
and $\zeta_0$ is an accumulation point of $\{x^{*}_{p+1},p\in\PP\}$.
We assume that $x^{*}_{p+1}\rightarrow\zeta_0$ for simplicity.
Now we prove that
\[
d_{i_{*},j_{*}}(\tens{W}^{(p)})\rightarrow0\ \
\text{and}\quad\omega_{i_{*},j_{*}}(\tens{W}^{(p)})\rightarrow0
\]
 when $p\in\PP$ tends to infinity,
and thus get (i) by the
continuity of \eqref{eq:multilinear_function-1}. 
 
Denote by
\begin{align*}
\vartheta_{p} \eqdef 2d_{i_{*},j_{*}}(\tens{W}^{(p)})(x^{*}_{p+1}-{(x^{*}_{p+1})}^3)
-\omega_{i_{*},j_{*}}(\tens{W}^{(p)}){(x^{*}_{p+1})}^2
\end{align*}
and
\begin{align*}
\matr{M}_{p} \eqdef
\begin{bmatrix}
&2(x^{*}_{p+1}-{(x^{*}_{p+1})}^3)
&- {(x^{*}_{p+1})}^2 \\
&1-6{(x^{*}_{p+1})}^2+{(x^{*}_{p+1})}^4
&- x^{*}_{p+1}+{(x^{*}_{p+1})}^3
\end{bmatrix}.
\end{align*}
By \eqref{equa-increase-h} and \eqref{eq-derivative-0},
we see that
\begin{align*}
\matr{M}_{p}
\begin{bmatrix}
d_{i_{*},j_{*}}(\tens{W}^{(p)})\\
\omega_{i_{*},j_{*}}(\tens{W}^{(p)})
\end{bmatrix}
=
\begin{bmatrix}
\vartheta_{p}\\
0
\end{bmatrix}.
\end{align*}
Note that
\[
\det (\matr{M}_{p})= -{(x_{p+1}^{*})}^2({(x^{*}_{p+1})}^2+1)^2
\rightarrow -\zeta_0^2(\zeta_0^2+1)^2 \neq 0\]
when $p\in\PP$ tends to infinity. 
Then $\matr{M}_{p}$ is invertible when $p$ is large enough.
Note that $\vartheta_{p}\rightarrow0$. 
It follows that
\begin{align*}
\begin{bmatrix}
d_{i_{*},j_{*}}(\tens{W}^{(p)})\\
\omega_{i_{*},j_{*}}(\tens{W}^{(p)})
\end{bmatrix}
=
\matr{M}_{p}^{-1}
\begin{bmatrix}
\vartheta_{p}\\
0
\end{bmatrix}\rightarrow0
\end{align*}
when $p\in\PP$ tends to infinity. 

\noindent(ii) Let $(i_{*},j_{*})$ be any pair.
There exists an accumulation point $\matr{Q}_{*}\in\ON{n}$ such that,
if $\{\matr{Q}_{\ell},\ell\in\LL\}$ is the subsequence of $\{\matr{Q}_{k},k\in\NN\}$ located in $\mathcal{N}(\matr{Q}_{*},\delta)$,
then $(i_{*},j_{*})$ appears for an infinite number of times in the sequence of pairs $\{(i_{\ell+1},j_{\ell+1}),\ell\in\LL\}$.\\
\indent (a) If it appears for an infinite number of times in
\[\{(i_{\ell+1},j_{\ell+1}),\matr{Q}_{\ell+1}\notin\mathcal{N}(\matr{Q}_{*},\delta)\},\]
then the result follows by the same reasoning as in (i).\\
\indent (b)
Otherwise,
it appears for an infinite number of times in
\[\{(i_{\ell+1},j_{\ell+1}),\matr{Q}_{\ell+1}\in\mathcal{N}(\matr{Q}_{*},\delta)\}.\]
We construct the subsequence
$\{\matr{Q}_{p},p\in\PP\}$ of $\{\matr{Q}_{\ell},\ell\in\LL\}$
such that
\[(i_{p+1},j_{p+1})=(i_{*},j_{*})\quad\text{and}\ \
\matr{Q}_{p+1}\in\mathcal{N}(\matr{Q}_{*},\delta).\]
Note that $\matr{Q}_{p}\rightarrow\matr{Q}_{*}$ and $\matr{Q}_{p+1}\rightarrow\matr{Q}_{*}$
 when $p\in\PP$ tends to infinity. 
We get that $x^{*}_{p+1}\rightarrow0$, and eventually from the proof of  Theorem~\ref{theorem-convergence-stationary}: 
\[
d_{i_{*},j_{*}}(\tens{W}^{(p)})\rightarrow0\quad\text{and}\quad\omega_{i_{*},j_{*}}(\tens{W}^{(p)})\geq0
\]
 when $p\in\PP$ is large enough. 
Then we prove (ii) by the
continuity of \eqref{eq:multilinear_function-1}. 

\noindent(iii) Note that there exist a finite number of accumulation points and $|x^{*}_{k}|\leq 1$ for any $k > 1$.
The sequence $\{x^{*}_{k},k\}$ has a finite number of accumulation points. 
Let $\zeta_0$ be any one of them.
Then $\zeta_0\in[-1,1]$.
There exists a subsequence $\LL\subseteq\NN$ such that
$x^{*}_{\ell+1}\rightarrow\zeta_0$
when $\ell\in\LL$ tends to infinity.

Next, we have that
\[
d_{i_{\ell+1},j_{\ell+1}}(\tens{W}^{(\ell)})\rightarrow0,
\]
which follows by \eqref{eq-derivative-0} if $\zeta_0 = 0$,
and by a reasoning similar to (i) if $\zeta_0 \neq 0$.
Finally, note that there exist a finite number of accumulation points in $\{x^{*}_{k}\}$, hence the proof is complete.
\end{proof}

\bibliographystyle{elsarticle-num}
\bibliography{tensor_classes}

\end{document}